\pdfoutput=1
\RequirePackage{ifpdf}
\ifpdf 
\documentclass[pdftex]{sigma}
\else
\documentclass{sigma}
\fi

\usepackage{mathtools}
\usepackage[matrix,arrow,curve,frame]{xy}
\usepackage{tikz} \usetikzlibrary{calc} \tikzset{>=latex} \usetikzlibrary{backgrounds}
\usepackage{tikz-cd}
\usetikzlibrary{arrows}
\usetikzlibrary{decorations.pathreplacing,calligraphy}

\tikzcdset{scale cd/.style={every label/.append style={scale=#1},
 cells={nodes={scale=#1}}}}

\numberwithin{equation}{section}

\newtheorem{Theorem}{Theorem}[section]
\newtheorem*{Theorem*}{Theorem}
\newtheorem{Corollary}[Theorem]{Corollary}
\newtheorem{Lemma}[Theorem]{Lemma}
\newtheorem{Proposition}[Theorem]{Proposition}

\theoremstyle{definition}
\newtheorem{Definition}[Theorem]{Definition}

\newtheorem{Example}[Theorem]{Example}
\newtheorem{Remark}[Theorem]{Remark}

\begin{document}

\allowdisplaybreaks

\newcommand{\arXivNumber}{2508.18895}

\renewcommand{\thefootnote}{}

\renewcommand{\PaperNumber}{012}

\FirstPageHeading

\ShortArticleName{A Tensor Category Construction of the $W_{p,q}$ Triplet Vertex Operator Algebra}

\ArticleName{A Tensor Category Construction\\ of the $\boldsymbol{W_{p,q}}$ Triplet Vertex Operator Algebra\\ and Applications\footnote{This paper is a~contribution to the Special Issue on Recent Advances in Vertex Operator Algebras in honor of James Lepowsky. The~full collection is available at \href{https://sigma-journal.com/Lepowsky.html}{https://sigma-journal.com/Lepowsky.html}}}

\Author{Robert MCRAE~$^{\rm a}$ and Valerii SOPIN~$^{\rm b}$}

\AuthorNameForHeading{R.~McRae and V.~Sopin}

\Address{$^{\rm a)}$~Yau Mathematical Sciences Center, Tsinghua University, Beijing 100084, P.R.~China}
\EmailD{\mail{rhmcrae@tsinghua.edu.cn}}

\Address{$^{\rm b)}$~Shanghai Institute for Mathematics and Interdisciplinary Sciences, \\
\hphantom{$^{\rm b)}$}~Shanghai 200433, P.R.~China}
\EmailD{\mail{vvs@myself.com}}

\ArticleDates{Received August 27, 2025, in final form January 19, 2026; Published online February 09, 2026}

\Abstract{For coprime $p,q\in\mathbb{Z}_{\geq 2}$, the triplet vertex operator algebra $W_{p,q}$ is a non-simple extension of the universal Virasoro vertex operator algebra of central charge $c_{p,q}=\smash{1-\frac{6(p-q)^2}{pq}}$, and it is a basic example of a vertex operator algebra appearing in logarithmic conformal field theory. Here, we give a new construction of $W_{p,q}$ different from the original screening operator definition of Feigin--Gainutdinov--Semikhatov--Tipunin. Using our earlier work on the tensor category structure of modules for the Virasoro algebra at central charge~$c_{p,q}$, we show that the simple modules appearing in the decomposition of $W_{p,q}$ as a module for the Virasoro algebra have $\mathrm{PSL}_2$-fusion rules and generate a symmetric tensor category equivalent to $\operatorname{Rep}\mathrm{PSL}_2$. Then we use the theory of commutative algebras in braided tensor categories to construct $W_{p,q}$ as an appropriate non-simple modification of the canonical algebra in the Deligne tensor product of $\operatorname{Rep}\mathrm{PSL}_2$ with this Virasoro subcategory. As a~consequence, we show that the automorphism group of $W_{p,q}$ is $\mathrm{PSL}_2(\mathbb{C})$. We~also define a braided tensor category \smash{$\mathcal{O}_{c_{p,q}}^0$} consisting of modules for the Virasoro algebra at central charge $c_{p,q}$ that induce to untwisted modules of $W_{p,q}$. We show that \smash{$\mathcal{O}_{c_{p,q}}^0$} tensor embeds into the $\mathrm{PSL}_2(\mathbb{C})$-equivariantization of the category of $W_{p,q}$-modules and is closed under contragredient modules. We conjecture that \smash{$\mathcal{O}_{c_{p,q}}^0$} has enough projective objects and is the correct category of Virasoro modules for constructing logarithmic minimal models in conformal field theory.}


\Classification{17B69; 18M15; 81R10; 81T40}

\begin{flushright}
\begin{minipage}{58mm}
\it Dedicated with admiration and\\
gratitude to James Lepowsky on\\ the occasion of his 80th birthday
\end{minipage}
\end{flushright}

\renewcommand{\thefootnote}{\arabic{footnote}}
\setcounter{footnote}{0}

\section{Introduction}

The triplet $W$-algebras $W_{p,q}$ for coprime $p,q\in\mathbb{Z}_{\geq 1}$ are fundamental examples of vertex operator algebras (VOAs) with finite but non-semisimple representation theory. When $q=1$, $W_{p,1}$ is a~simple and $C_2$-cofinite VOA with $2p$ simple modules \cite{AM-Wp1-alg}, automorphism group $\mathrm{PSL}_2(\mathbb{C})$~\cite{ALM}, and a non-semisimple modular tensor category of representations \cite{GaNe,TW1}. Moreover, as conjectured in \cite{FGST0} and proved recently in \cite{CLR, CLM,GaNe}, this modular tensor category is equivalent to the category of finite-dimensional representations of a quasi-Hopf modification of the restricted quantum group of $\mathfrak{sl}_2$ at the root of unity ${\rm e}^{\pi {\rm i}/p}$. When $p,q\geq 2$, $W_{p,q}$ is no longer simple but is still $C_2$-cofinite \cite{AM-W2p-alg, TW2}. These more complicated VOAs are the subject of this paper. Our main results are a~new construction of $W_{p,q}$ for $p,q\geq 2$ using tensor category methods and a~proof that, as in the $q=1$ case, the automorphism group of $W_{p,q}$ is $\mathrm{PSL}_2(\mathbb{C})$.

We first review what has already been established about $W_{p,q}$ for $p,q\geq 2$. The triplet algebra~$W_{p,q}$ was constructed in \cite{FGST1-Log} as the intersection of the kernels of two screening operators acting on the lattice VOA $V_{\sqrt{2pq}\mathbb{Z}}$. It has central charge \smash{$c_{p,q}=1-\frac{6(p-q)^2}{pq}$}, contains the universal Virasoro VOA $V_{c_{p,q}}$ of central charge $c_{p,q}$ as a subalgebra, and is non-simple with the rational simple Virasoro VOA $L_{c_{p,q}}$ of central charge $c_{p,q}$ as its unique simple quotient. Just as $L_{c_{p,q}}$ is the VOA of minimal models in rational conformal field theory in physics, $W_{p,q}$ is the VOA of $W$-extended logarithmic minimal models in logarithmic conformal field theory \cite{PR, Ras-W-ext, Ras-Fusion,RP1, GRW1, GRW2}. Such logarithmic conformal field theories sometimes arise in the analysis of statistical models at critical points; in particular the $(p,q)=(2,3)$ case is relevant for critical percolation (see, for example, \cite{CR} and references therein).

The VOA structure and representation theory of $W_{p,q}$ for $q\geq 2$ are not yet as well understood as in the $q=1$ case. However, the decomposition of $W_{p,q}$ as a module for the Virasoro algebra~$\mathcal{V}{\rm ir}$ is already known \cite{AM-W2p-alg, AM-C2-W-alg,FGST1-Log, TW2}
\begin{gather}\label{eqn:Wpq-decomp-intro}
W_{p,q} = V_{c_{p,q}} \oplus \sum_{n=2}^\infty (2n-1)\cdot \mathcal{L}_{2np-1,1},
\end{gather}
where $\mathcal{L}_{2np-1,1}$ is the simple $\mathcal{V}{\rm ir}$-module of central charge $c_{p,q}$ and lowest conformal weight $(np-1)(nq-1)$. This decomposition follows from the fact that screening operators acting on~$V_{\sqrt{2pq}\mathbb{Z}}$ and its modules form almost exact Felder complexes \cite{Felder} of Feigin--Fuchs modules for~$\mathcal{V}{\rm ir}$~\cite{Feigin--Fuchs}, along with the detailed socle series structure of Feigin--Fuchs modules. In particular, $W_{p,q}$, as the intersection of the kernel of two screening operators acting on $V_{\sqrt{2pq}\mathbb{Z}}$, is generated by the socles of the Feigin--Fuchs $\mathcal{V}{\rm ir}$-modules that make up $V_{\sqrt{2pq}\mathbb{Z}}$, together with the vacuum. It is not clear whether it is possible to describe this intersection of kernels of screening operators without using the rather technical structural results on Feigin--Fuchs modules.

Another known result on the structure of $W_{p,q}$ is that it is a $C_2$-cofinite VOA (see \cite{AM-W2p-alg} for the~${q=2}$ case and \cite{TW2} in general), which thanks to \cite{Hu-C2} implies that its representation category is a finite abelian braided monoidal category. Simple $W_{p,q}$-modules have been classified in \cite{AM-W2p-alg-2} for the $q=2$ case and \cite{TW2} in general, while logarithmic indecomposable $W_{p,q}$-modules have been constructed in \cite{AM-Exp-log, Nak-Wpq-projective}. The monoidal structure on the category of $W_{p,q}$-modules is not rigid, essentially because $W_{p,q}$ is not simple or self-contragredient when $q\geq 2$, but fusion rules have been obtained by various methods in, for example, \cite{Nak-Wpq-fusion,Ras-Fusion, RW-Mod, Wood-Wpq-fusion}.

Now we review conjectures on $W_{p,q}$. In~\cite{FGST2-KL-dual}, a relation (though not quite an equivalence) was conjectured between the monoidal categories of modules for $W_{p,q}$ and for a certain ``Kazhdan--Lusztig dual'' quantum group, generalizing the previously-mentioned equivalence between the categories of modules for $W_{p,1}$ and for the restricted quantum group of~$\mathfrak{sl}_2$. Proving a precise version of this conjecture remains one of the most significant open problems pertaining to the~$W_{p,q}$ triplet algebra. Another problem is determining the automorphism group of $W_{p,q}$, which is $\mathrm{PSL}_2(\mathbb{C})$ in the $q=1$ case~\cite{ALM}. It was suggested in \cite{FGST1-Log} that $\mathrm{PSL}_2(\mathbb{C})$ should also act by VOA automorphisms of $W_{p,q}$ in the $q\geq 2$ case. However, although two derivations of $W_{p,q}$ labeled $E$ and $F$ were constructed in \cite{TW2}, it was not checked there whether the exponentials of~$E$ and $F$ actually generate an action of $\mathrm{PSL}_2(\mathbb{C})$ by automorphisms.

 In the present paper, we give a new construction of $W_{p,q}$ that is independent of the original construction in \cite{FGST1-Log}. In particular, we use tensor category methods, rather than analysis of screening operators on $V_{\sqrt{2pq}\mathbb{Z}}$, to show directly that the $\mathcal{V}{\rm ir}$-module direct sum on the right-hand side of \eqref{eqn:Wpq-decomp-intro} admits the structure of a VOA. Moreover, we show that the VOA structure on the $\mathcal{V}{\rm ir}$-module direct sum in \eqref{eqn:Wpq-decomp-intro} is
sufficiently unique, so that we can conclude that the VOA we have constructed is isomorphic to the triplet algebra $W_{p,q}$ constructed in \cite{FGST1-Log}. As the main application of our construction, we prove the previously-mentioned conjecture that the automorphism group of $W_{p,q}$ is $\mathrm{PSL}_2(\mathbb{C})$, with fixed-point subalgebra $V_{c_{p,q}}$. This will follow from the fact (which will be obvious from our construction) that the multiplicity spaces of the indecomposable $\mathcal{V}{\rm ir}$-module direct summands in \eqref{eqn:Wpq-decomp-intro} carry irreducible representations of~$\mathrm{PSL}(2,\mathbb{C})$ which are suitably compatible with the vertex operator that we construct on the direct sum in \eqref{eqn:Wpq-decomp-intro}. Our result that $V_{c_{p,q}}$ is the fixed-point subalgebra of the $\mathrm{PSL}(2,\mathbb{C})$-action on~$W_{p,q}$ has implications for the relationship between the representation theories of~$W_{p,q}$ and the Virasoro algebra (similar to~\cite{McR-G-equiv}) that we will explore in Section~\ref{sec:Vir-trip-relations}. We discuss our results and methods in more detail next.

Our starting point is our previous paper \cite{MS-cpq-Vir} where we detailed some of the tensor structure of the category $\mathcal{O}_{c_{p,q}}$ of $C_1$-cofinite modules for the universal Virasoro VOA $V_{c_{p,q}}$ of central charge~$c_{p,q}$. This is the same as the category of finite-length modules for the Virasoro algebra~$\mathcal{V}{\rm ir}$ of central charge~$c_{p,q}$ whose composition factors are irreducible quotients of reducible Verma modules, and it is a non-rigid braided tensor category~\cite{CJORY}.
By \eqref{eqn:Wpq-decomp-intro}, $W_{p,q}$ is an infinite direct sum of modules in $\mathcal{O}_{c_{p,q}}$, which by \cite{CMY1,HKL} implies that $W_{p,q}$ has the structure of a commutative algebra in the ind-completion (or direct limit completion) $\operatorname{Ind}(\mathcal{O}_{c_{p,q}})$ of the braided tensor category~$\mathcal{O}_{c_{p,q}}$.
Our goal is to use the tensor structure of~$\mathcal{O}_{c_{p,q}}$ to construct an at first possibly different commutative algebra structure (with $\mathrm{PSL}_2(\mathbb{C})$ automorphism group) on the~$\mathcal{V}{\rm ir}$-module direct sum in \eqref{eqn:Wpq-decomp-intro}, without assuming that \eqref{eqn:Wpq-decomp-intro} already admits such a commutative algebra or VOA structure. Then by \cite{HKL}, any commutative algebra structure that we obtain on the right-hand side of \eqref{eqn:Wpq-decomp-intro} is equivalent to some VOA structure, which we must then show is isomorphic to the already-known VOA $W_{p,q}$ using a suitable uniqueness result.

To achieve this, we first show in Section~\ref{sec:Vir-fus-rules} that the $\mathcal{V}{\rm ir}$-module direct summands in \eqref{eqn:Wpq-decomp-intro}, except with the non-simple VOA $V_{c_{p,q}}$ replaced by its contragredient module $V_{c_{p,q}}'$, are closed under the fusion tensor product of $\mathcal{O}_{c_{p,q}}$. Moreover, we show that these summands have the same fusion rules as the simple modules in the category $\operatorname{Rep} \mathrm{PSL}_2$ of finite-dimensional continuous $\mathrm{PSL}_2(\mathbb{C})$-modules.
This fusion rule calculation is subtle since $\mathcal{O}_{c_{p,q}}$ is non-semisimple and in particular contains logarithmic modules (on which the Virasoro operator $L_0$ acts non-semisimply). To show in Theorem~\ref{thm:Vir-fus-rules} below that the fusion tensor product $\mathcal{L}_{2mp-1,1}\boxtimes\mathcal{L}_{2np-1,1}$ in $\mathcal{O}_{c_{p,q}}$ is essentially semisimple (and in particular not logarithmic), we use fusion rules from~\cite{MS-cpq-Vir} that show $\mathcal{L}_{2mp-1,1}\boxtimes\mathcal{L}_{2np-1,1}$ is a quotient of the ``Kac module'' $\mathcal{K}_{2mp-1,2np-1}$, which is a~certain finite-length submodule of a Feigin--Fuchs $\mathcal{V}{\rm ir}$-module \cite{MDRR}. This shows the fusion tensor product is non-logarithmic, but to calculate it precisely, we need further analysis using the detailed structure of $\mathcal{K}_{2mp-1,2np-1}$ to show that $\mathcal{L}_{2mp-1,1}\boxtimes\mathcal{L}_{2np-1,1}$ is precisely the top socle series layer of $\mathcal{K}_{2mp-1,2np-1}$ (except in the $m=n$ case, where a direct summand of $V_{c_{p,q}}'$ rather than its simple quotient appears). It is interesting that even though our tensor category construction of $W_{p,q}$ is independent of screening operators, it is not independent of the socle series structure of Feigin--Fuchs modules, since these inform the structure of $\mathcal{K}_{2mp-1,2np-1}$ that we use in our computation of $\mathcal{L}_{2mp-1,1}\boxtimes\mathcal{L}_{2np-1,1}$.

Motivated by these Virasoro fusion rules, we then in Section~\ref{sec:sl2-type-cat} define $\mathcal{C}_{\mathrm{PSL}_2}$ to be the full subcategory of $\mathcal{O}_{c_{p,q}}$ consisting of finite direct sums of the modules $V_{c_{p,q}}'$ and $\mathcal{L}_{2np-1,1}$ for $n\in\mathbb{Z}_{\geq 2}$. We show that the fusion tensor product on $\mathcal{O}_{c_{p,q}}$ gives $\mathcal{C}_{\mathrm{PSL}_2}$ the structure of a rigid symmetric tensor category with new unit object $V_{c_{p,q}}'$ rather than $V_{c_{p,q}}$, which is not an object of $\mathcal{C}_{\mathrm{PSL}_2}$. Moreover, we show that $\mathcal{C}_{\mathrm{PSL}_2}$ is symmetric tensor equivalent to $\operatorname{Rep} \mathrm{PSL}_2$. The most interesting part of the proof is showing that $\mathcal{C}_{\mathrm{PSL}_2}$ is rigid, since its objects are not rigid as objects of $\mathcal{O}_{c_{p,q}}$ (because the unit objects in $\mathcal{O}_{c_{p,q}}$ and $\mathcal{C}_{\mathrm{PSL}_2}$ are different). We prove that $\mathcal{C}_{\mathrm{PSL}_2}$ is rigid by showing that a certain $F$-matrix entry (or $6j$-symbol) associated to a suitable associativity isomorphism is non-zero, and we show this using constraints on the $F$-matrix coming from the hexagon axiom for braided tensor categories. This is partially related to (but in this specific case much simpler than) Huang's proof of rigidity for the module category of a rational $C_2$-cofinite VOA \cite{Hu-rigidity}, where relevant $F$-matrix entries were shown to be non-zero using information coming from modular character transformations.

Alternatively, one could use the $\operatorname{Rep} \mathrm{PSL}_2$ fusion rules of $\mathcal{C}_{\mathrm{PSL}_2}$ together with the recent paper~\cite{EP} to show that $\mathcal{C}_{\mathrm{PSL}_2}$ is rigid. However,
we believe that constraining $F$-matrices using the hexagon axiom may be useful for proving rigidity in further examples of vertex algebraic tensor categories where the fusion rules are not fully known. We also remark that calculating $F$-matrices (or $6j$-symbols) is an interesting problem in its own right since $F$-matrices appear directly in rational conformal field theory \cite{MS} and state-sum invariants of $3$-manifolds such as Reshetikhin--Turaev \cite{RT} and Turaev--Viro \cite{TV} invariants, and they are also related to Kashaev's volume conjecture for hyperbolic knots \cite{Ka}. They are also relevant in the classification of semisimple tensor categories with given fusion rules \cite{RSW}, since the information in the $6j$-symbols is precisely what is forgotten in passing from a semisimple tensor category to its Grothendieck ring.

In Section~\ref{sec:comm-alg-and-VOA}, we complete our tensor-categorical construction of $W_{p,q}$. First, since $\mathcal{C}_{\mathrm{PSL}_2}$ and $\operatorname{Rep} \mathrm{PSL}_2$ are symmetric tensor equivalent, the canonical algebra of $\operatorname{Rep} \mathrm{PSL}_2$ induces a~simple commutative algebra $W_{p,q}'$ in the Deligne product category $\operatorname{Ind}(\operatorname{Rep} \mathrm{PSL}_2\otimes\mathcal{C}_{\mathrm{PSL}_2})$ with decomposition
\begin{gather}\label{eqn:Wpq'-decomp-intro}
W_{p,q}'=V_{c_{p,q}}'\oplus\bigoplus_{n=2}^\infty V_{2n-2}\otimes\mathcal{L}_{2np-1,1},
\end{gather}
where $V_{2n-2}$ is the $(2n-1)$-dimensional simple continuous $\mathrm{PSL}_2(\mathbb{C})$-module. For information about canonical algebras, see, for example, \cite{EGNO}, or \cite{CKM2} for a more detailed exposition, where canonical algebras are constructed by ``gluing'' braid-reversed equivalent braided tensor categories. We then show that applying the forgetful fiber functor $\operatorname{Rep} \mathrm{PSL}_2\rightarrow\mathcal{V}ec$ to \eqref{eqn:Wpq'-decomp-intro} yields a simple commutative algebra in $\operatorname{Ind}(\mathcal{C}_{\mathrm{PSL}_2})$ with automorphism group $\mathrm{PSL}_2(\mathbb{C})$ and with the same $\mathcal{V}{\rm ir}$-module decomposition as \eqref{eqn:Wpq-decomp-intro}, except that $V_{c_{p,q}}$ is replaced by $V_{c_{p,q}}'$.
Next, we explain how to use the unique (up to scaling) non-zero $\mathcal{V}{\rm ir}$-homomorphism $V_{c_{p,q}}'\rightarrow V_{c_{p,q}}$ to turn the simple commutative algebra structure on $W_{p,q}'$ into a commutative algebra structure on the direct sum on the right-hand side of \eqref{eqn:Wpq-decomp-intro} with a unique simple ideal and simple quotient.

We then have to show that the commutative algebra structure we have obtained on the right-hand side of \eqref{eqn:Wpq-decomp-intro} yields the same VOA $W_{p,q}$ of \cite{FGST1-Log} under the correspondence between commutative algebras and VOAs from \cite{HKL}. Thus we show that $W_{p,q}'$ is the unique simple commutative algebra structure on the direct sum in \eqref{eqn:Wpq'-decomp-intro} and use this to prove that the algebra structure on \eqref{eqn:Wpq-decomp-intro} is also unique if it has a suitable simple ideal. Since the triplet algebra~$W_{p,q}$ has such a simple ideal, it follows that the algebra we have constructed is indeed the triplet algebra, and moreover, we can show that the automorphism group of $W_{p,q}$ is $\mathrm{PSL}_2(\mathbb{C})$ by transfer from~$W_{p,q}'$. These results are formally stated in Theorem~\ref{thm1} and Corollary~\ref{cor:Wpq-aut-gp}.

Finally, in Section~\ref{sec:Vir-trip-relations}, we use the VOA extension theory of \cite{CKM1, CMY1,HKL} to discuss tensor-categorical relations between the Virasoro category $\mathcal{O}_{c_{p,q}}$ and the triplet category $\operatorname{Rep}(W_{p,q})$. Using the induction tensor functor $F$ from $\mathcal{O}_{c_{p,q}}$ to the category of ``non-local'', or ``twisted'', $W_{p,q}$-modules, we define the category $\mathcal{O}_{c_{p,q}}^0$ to be the full subcategory of $\mathcal{O}_{c_{p,q}}$ consisting of modules which induce to ordinary modules for the triplet algebra, that is, objects of $\operatorname{Rep}(W_{p,q})$. We then show that when restricted to this subcategory, $F$ defines a fully faithful tensor functor from \smash{$\mathcal{O}_{c_{p,q}}^0$} to the $\mathrm{PSL}_2(\mathbb{C})$-equivariantization of $\operatorname{Rep}(W_{p,q})$. However, unlike in the $q=1$ case considered in \cite[Section~7]{MY-cp1-Vir}, induction does not give an equivalence between \smash{$\mathcal{O}_{c_{p,q}}^0$} and the equivariantization \smash{$\operatorname{Rep}(W_{p,q})^{\mathrm{PSL}_2(\mathbb{C})}$} because $W_{p,q}$ is not simple.

The main result of Section~\ref{sec:Vir-trip-relations} is that $\mathcal{O}_{c_{p,q}}^0$ contains all simple objects of $\mathcal{O}_{c_{p,q}}$ and is a ribbon Grothendieck--Verdier category in the sense of \cite{BD}, where the Grothendieck--Verdier duality (a~weaker duality structure than rigidity) is given by contragredient modules (see \cite{ALSW}). The main difficulty is to show that \smash{$\mathcal{O}_{c_{p,q}}^0$} is actually closed under taking contragredients. We conjecture that \smash{$\mathcal{O}_{c_{p,q}}^0$}, unlike the larger category \smash{$\mathcal{O}_{c_{p,q}}$}, has enough projective objects. If so, then it is natural to conjecture that \smash{$\mathcal{O}_{c_{p,q}}^0$} is the correct category of $\mathcal{V}{\rm ir}$-modules to use to build a full logarithmic minimal model conformal field theory.

\section{Some Virasoro fusion rules}\label{sec:Vir-fus-rules}

To give a novel construction of the $W_{p,q}$ triplet VOA, without using screening operators \cite{FGST1-Log, FGST2-KL-dual}, we first need to calculate the fusion rules for the Virasoro modules which appear in the decomposition of $W_{p,q}$ as a Virasoro module. Thus let $\mathcal{V}{\rm ir}$ be the Virasoro Lie algebra with basis $\lbrace L_n\mid n\in\mathbb{Z}\rbrace\cup\lbrace\mathbf{c}\rbrace$, where $\mathbf{c}$ is central and
\begin{gather*}
[L_m,L_n]=(m-n)L_{m+n} +\frac{m^3-m}{12}\delta_{m+n,0}\mathbf{c}
\end{gather*}
for $m,n\in\mathbb{Z}$. In this paper, we will only consider $\mathcal{V}{\rm ir}$-modules on which $\mathbf{c}$ acts by the central charge
\smash{$
c_{p,q}=1-\frac{6(p-q)^2}{pq}
$}
for some coprime $p,q\in\mathbb{Z}_{\geq 2}$.

Let $V_{c_{p,q}}$ be the universal Virasoro VOA of central charge $c_{p,q}$ \cite{Frenkel-YZhu}. By \cite{CJORY}, the category~$\mathcal{O}_{c_{p,q}}$ of $C_1$-cofinite $V_{c_{p,q}}$-modules admits the vertex algebraic braided tensor category structure of~\cite{HLZ1,HLZ2,HLZ3,HLZ4,HLZ5,HLZ6,HLZ7,HLZ8}. More specifically, $\mathcal{O}_{c_{p,q}}$ is the category of finite-length $\mathcal{V}{\rm ir}$-modules of central charge $c_{p,q}$ whose composition factors are of the form $\mathcal{L}_{r,s}$ for some $r,s\in\mathbb{Z}_{\geq 1}$. Here, $\mathcal{L}_{r,s}$ is the simple highest-weight $\mathcal{V}{\rm ir}$-module of central charge $c_{p,q}$ whose lowest $L_0$-eigenvalue is
\begin{gather*}
h_{r,s} =\frac{r^2-1}{4}\cdot\frac{q}{p}-\frac{rs-1}{2}+\frac{s^2-1}{4}\cdot\frac{p}{q}.
\end{gather*}
Note the symmetries $h_{r,s}=h_{r+p,s+q}$ and $h_{r,s}=h_{-r-s}$, which imply that each simple object of~$\mathcal{O}_{c_{p,q}}$ is isomorphic to a unique $\mathcal{L}_{r,s}$ such that $r\in\mathbb{Z}_{\geq 1}$, $1\leq s\leq q$, and $qr\geq ps$.

In \cite{MS-cpq-Vir}, we determined some of the braided tensor category structure of $\mathcal{O}_{c_{p,q}}$, focusing mainly on the Virasoro Kac modules $\mathcal{K}_{r,s}$ defined in \cite{MDRR}. For each $r,s\in\mathbb{Z}_{\geq 1}$, $\mathcal{K}_{r,s}$ is the submodule of a Feigin--Fuchs module \cite{Feigin--Fuchs} (of lowest $L_0$-eigenvalue $h_{r,s}$) generated by all vectors with $L_0$-eigenvalue strictly less than $h_{r,s}+rs$. Thus unlike the simple modules $\mathcal{L}_{r,s}$, we do not necessarily have $\mathcal{K}_{r,s} = \mathcal{K}_{r',s'}$ if $h_{r,s}=h_{r',s'}$; instead, $\mathcal{K}_{r,s}$ is a proper submodule of $\mathcal{K}_{r+p,s+q}$, for example.
As special cases, there are non-split short exact sequences
\begin{gather}
0\longrightarrow \mathcal{L}_{(m+2)p-r,1} \longrightarrow \mathcal{K}_{mp+r,1} \longrightarrow \mathcal{L}_{mp+r,1} \longrightarrow 0,\nonumber\\
0\longrightarrow \mathcal{L}_{1,(n+2)q-s}\longrightarrow \mathcal{K}_{1,nq+s} \longrightarrow \mathcal{L}_{1,nq+s}\longrightarrow 0\label{eqn:Kac-1}
\end{gather}
for $m,n\in\mathbb{Z}_{\geq 0}$, $1\leq r\leq p-1$, and $1\leq s\leq q-1$.
As a $\mathcal{V}{\rm ir}$-module, the universal Virasoro VOA $V_{c_{p,q}}$ is isomorphic to $\mathcal{K}_{1,1}$.
In particular, by \eqref{eqn:Kac-1}, the unique simple $\mathcal{V}{\rm ir}$-module quotient of $V_{c_{p,q}}$ is $\mathcal{L}_{1,1}$, and the unique simple $\mathcal{V}{\rm ir}$-submodule is $\mathcal{L}_{2p-1,1}=\mathcal{L}_{1,2q-1}$.

In \cite{MS-cpq-Vir}, we used Belavin--Polyakov--Zamolodchikov differential equations to show that $\mathcal{K}_{1,2}$ is rigid and self-dual in the tensor category $\mathcal{O}_{c_{p,q}}$, which implies in particular that the tensoring functors $\mathcal{K}_{1,2}\boxtimes\bullet$ and $\bullet\boxtimes\mathcal{K}_{1,2}$ are exact (see, for example, \cite[Proposition 4.2.1]{EGNO}), and that
\begin{gather}\label{eqn:rigid_hom_iso}
\operatorname{Hom}_{\mathcal{V}{\rm ir}}(\mathcal{K}_{1,2}\boxtimes W_1, W_2)\cong\operatorname{Hom}_{\mathcal{V}{\rm ir}}(W_1,\mathcal{K}_{1,2}\boxtimes W_2)
\end{gather}
for any modules $W_1$, $W_2$ in \smash{$\mathcal{O}_{c_{p,q}}$} (see, for example, \cite[Proposition 2.10.8]{EGNO}). Due to the symmetry~${c_{p,q}=c_{q,p}}$, the same results hold for $\mathcal{K}_{2,1}$.

In \cite{MS-cpq-Vir}, we also completely determined the fusion products $\mathcal{K}_{1,2}\boxtimes\mathcal{K}_{r,s}$ in $\mathcal{O}_{c_{p,q}}$ for all $r,s\in\mathbb{Z}_{\geq 1}$. In particular, there is a short exact sequence
 \begin{gather}\label{eqn:K12xKrs}
 0\longrightarrow\mathcal{K}_{r,s-1}\longrightarrow\mathcal{K}_{1,2}\boxtimes\mathcal{K}_{r,s}\longrightarrow\mathcal{K}_{r,s+1}\longrightarrow 0,
 \end{gather}
which splits if and only if $q\nmid s$. In our proof of this result, the surjection $\mathcal{K}_{1,2}\boxtimes\mathcal{K}_{r,s}\rightarrow\mathcal{K}_{r,s+1}$ comes from the intertwining operator of type \smash{$\binom{\mathcal{K}_{r,s+1}}{\mathcal{K}_{1,2} \mathcal{K}_{r,s}}$} obtained by restricting a Heisenberg Fock module intertwining operator involving the corresponding Feigin--Fuchs modules to their Kac submodules. The existence of a non-zero map $\mathcal{K}_{r,s-1}\rightarrow\mathcal{K}_{1,2}\boxtimes\mathcal{K}_{r,s}$ then follows from the~${W_1=\mathcal{K}_{r,s-1}}$, $W_2=\mathcal{K}_{r,s}$ case of \eqref{eqn:rigid_hom_iso}. It is not easy to show that the resulting sequence of maps \eqref{eqn:K12xKrs} is exact, but once this is done, it is then clear that $\mathcal{K}_{1,2}\boxtimes\mathcal{K}_{r,s}\cong\mathcal{K}_{r,s-1}\oplus\mathcal{K}_{r,s+1}$ when $q\nmid s$, because in this case the conformal weights of $\mathcal{K}_{r,s-1}$ and $\mathcal{K}_{r,s+1}$ are non-congruent mod $\mathbb{Z}$ and thus a non-trivial extension is impossible. When $q\mid s$, on the other hand, it turns out that $\mathcal{K}_{1,2}\boxtimes\mathcal{K}_{r,s}$ is a logarithmic extension of $\mathcal{K}_{r,s+1}$ by $\mathcal{K}_{r,s-1}$. See \cite{MS-cpq-Vir} for more details.

 We also showed in \cite[Theorem 6.7]{MS-cpq-Vir} that
\begin{gather}\label{eqn:Kr1xK1s}
\mathcal{K}_{r,1}\boxtimes\mathcal{K}_{1,s}\cong\mathcal{K}_{r,s}
\end{gather}
for all $r,s\in\mathbb{Z}_{\geq 1}$. The map $\mathcal{K}_{r,1}\boxtimes\mathcal{K}_{1,s}\rightarrow\mathcal{K}_{r,s}$ again comes from restricting a Heisenberg Fock module intertwining operator involving Feigin--Fuchs modules to their Kac submodules, though the proof that it is an isomorphism is non-trivial and uses the exact sequence \eqref{eqn:K12xKrs}. The exact sequence \eqref{eqn:K12xKrs} can also be used to compute $\mathcal{K}_{1,2}\boxtimes\mathcal{L}_{r,s}$ for all $r,s\in\mathbb{Z}_{\geq 1}$, since every simple module $\mathcal{L}_{r,s}$ has a resolution by Kac modules (this follows from \eqref{eqn:Kac-1} in the $s=1$ and $r=1$ cases); see \cite[Theorem 6.8]{MS-cpq-Vir} for details.

As we will discuss in more detail later, the simple $\mathcal{V}{\rm ir}$-submodules of the triplet vertex algebra~$W_{p,q}$ will come from among the modules $\mathcal{L}_{np-1,1}=\mathcal{L}_{1,nq-1}$ for $n\in\mathbb{Z}_{\geq 2}$ (actually, for the~$\mathcal{V}{\rm ir}$-submodules of $W_{p,q}$, we will only need $n$ even). At irrational central charges, the fusion rules for such simple modules were determined in \cite{Frenkel-MZhu} using Zhu algebra methods, while at central charge~$c_{1,q}$, their fusion rules were determined in \cite{MY-cp1-Vir} using the action of ${\rm SU}(2)$ by automorphisms on the doublet algebra, an abelian intertwining algebra extension of the triplet algebra $W_{1,q}$. But at central charge $c_{p,q}$ with $p\geq 2$, Zhu algebra computations seem to be difficult, and it is not \textit{a priori} clear that ${\rm SU}(2)$ acts on the triplet or doublet algebra. So instead, we prove the fusion rules using properties of Virasoro Kac modules derived in~\cite{MS-cpq-Vir}:

\begin{Theorem}\label{thm:Vir-fus-rules}
In the tensor category $\mathcal{O}_{c_{p,q}}$ for coprime $p,q\in\mathbb{Z}_{\geq 2}$,
\begin{gather*}
\mathcal{L}_{mp-1,1}\boxtimes\mathcal{L}_{np-1,1}\cong\bigoplus_{\substack{i=\vert m-n\vert +2\\
i+m+n\equiv 0 \mod 2}}^{m+n-2} \mathcal{L}_{ip-1,1}
\end{gather*}
for $m,n\in\mathbb{Z}_{\geq 2}$ such that $m\neq n$, while
\begin{gather*}
\mathcal{L}_{np-1,1}\boxtimes\mathcal{L}_{np-1,1}\cong\mathcal{K}_{1,1}'\oplus\bigoplus_{j=2}^{n-1} \mathcal{L}_{2jp-1,1}
\end{gather*}
for $n\in\mathbb{Z}_{\geq 2}$, where $\mathcal{K}_{1,1}'$ is the contragredient dual of $\mathcal{K}_{1,1}$.
\end{Theorem}
\begin{proof}
Since the tensor product on $\mathcal{O}_{c_{p,q}}$ is commutative, and since the tensor product formulas in the theorem statement are symmetric in $m$ and $n$, we assume throughout the proof that $m\geq n$. We will use the short exact sequences
\begin{align*}
0 \longrightarrow \mathcal{L}_{mp+1, 1} & \longrightarrow \mathcal{K}_{mp-1, 1} \longrightarrow \mathcal{L}_{mp-1, 1} \longrightarrow 0,\\
0 \longrightarrow \mathcal{L}_{1, nq+1} & \longrightarrow \mathcal{K}_{1, nq-1} \longrightarrow \mathcal{L}_{1, nq-1} \longrightarrow 0,
\end{align*}
which are special cases of \eqref{eqn:Kac-1}. Equation \eqref{eqn:Kac-1} also gives surjections $\mathcal{K}_{mp+1,1}\twoheadrightarrow\mathcal{L}_{mp+1,1}$ and $\mathcal{K}_{1,nq+1}\twoheadrightarrow\mathcal{L}_{1,nq+1}$. So noting that $\mathcal{L}_{1,nq-1}=\mathcal{L}_{np-1,1}$, we have right exact sequences
\begin{align*}
\begin{split}
&\mathcal{K}_{mp+1, 1} \longrightarrow \mathcal{K}_{mp-1, 1} \longrightarrow \mathcal{L}_{mp-1, 1} \longrightarrow 0,\\
&\mathcal{K}_{1, nq+1} \longrightarrow \mathcal{K}_{1, nq-1} \longrightarrow \mathcal{L}_{np-1,1} \longrightarrow 0.
\end{split}
\end{align*}
As the tensor product on $\mathcal{O}_{c_{p,q}}$ is right exact, we can tensor these two right exact sequences into a single right exact sequence using \eqref{eqn:Kr1xK1s}:
\begin{gather}\label{eqn:right-exact}
\mathcal{K}_{mp+1,nq-1}\oplus\mathcal{K}_{mp-1,nq+1}\longrightarrow\mathcal{K}_{mp-1,nq-1}\longrightarrow\mathcal{L}_{mp-1,1}\boxtimes\mathcal{L}_{np-1,1}\longrightarrow 0
\end{gather}
(see, for example, \cite[Lemma 2.5]{McR-Deligne}).

In particular, $\mathcal{L}_{mp-1,1}\boxtimes\mathcal{L}_{np-1,1}$ is some quotient of $\mathcal{K}_{mp-1,nq-1}$. Using the diagrams in \cite[Section 2.3]{MS-cpq-Vir}, as well as the conformal weight symmetries $h_{p+r,q+s}=h_{r,s}$ and $h_{-r,-s}=h_{r,s}$ for~${r,s\in\mathbb{Z}}$, $\mathcal{K}_{mp-1,nq-1}$ has the following composition series structure:
\begin{gather}\label{eqn:K-diag}
\begin{matrix}
\begin{tikzpicture}[->,>=latex,scale=1.2]
\node at (-5,3.75){$\mathcal{K}_{mp-1,nq-1}$:};
\node(m0) at (0,7.5){$\mathcal{L}_{1,(m-n)q+1}$};
\node(s1) at (-2,6.5){$\mathcal{L}_{p-1,(m-n+1)q+1}$};
\node(t1) at (2,6.5){$\mathcal{L}_{(m-n+2)p-1,1}$};
\node(m1) at (-2,5.5){$\mathcal{L}_{(m-n+2)p+1,1}$};
\node(m2) at (2,5.5){$\mathcal{L}_{1,(m-n+2)q+1}$};
\node(s2) at (-2,4.5){$\mathcal{L}_{p-1,(m-n+3)q+1}$};
\node(t2) at (2,4.5){$\mathcal{L}_{(m-n+4)p-1,1}$};
\node(m3) at (-2,3.5){$\mathcal{L}_{(m-n+4)p+1,1}$};
\node(m4) at (2,3.5){$\mathcal{L}_{1,(m-n+4)q+1}$};
\node(d1) at (-2,3){$\vdots$};
\node(d1') at (-.9,2.7){};
\node(d2) at (2,3){$\vdots$};
\node(d2') at (.9,2.7){};
\node(s3) at (-2,2){$\mathcal{L}_{p-1,(m+n-3)q+1}$};
\node(t3) at (2,2){$\mathcal{L}_{(m+n-2)p-1,1}$};
\node(m5) at (-2,1){$\mathcal{L}_{(m+n-2)p+1,1}$};
\node(m6) at (2,1){$\mathcal{L}_{1,(m+n-2)q+1}$};
\node(s4) at (0,0){$\mathcal{L}_{p-1,(m+n-1)q+1}$.};
\draw[] (t1) -- node[left]{} (m0);
\draw[] (t1) -- node[left]{} (m1);
\draw[] (t1) -- node[left]{} (m2);
\draw[] (t2) -- node[left]{} (m1);
\draw[] (t2) -- node[left]{} (m2);
\draw[] (t2) -- node[left]{} (m3);
\draw[] (t2) -- node[left]{} (m4);
\draw[] (t3) -- node[left]{} (d1');
\draw[] (t3) -- node[left]{} (d2);
\draw[] (t3) -- node[left]{} (m5);
\draw[] (t3) -- node[left]{} (m6);
\draw[] (m0) -- node[left]{} (s1);
\draw[] (m1) -- node[left]{} (s1);
\draw[] (m2) -- node[left]{} (s1);
\draw[] (m1) -- node[left]{} (s2);
\draw[] (m2) -- node[left]{} (s2);
\draw[] (m3) -- node[left]{} (s2);
\draw[] (d1) -- node[left]{} (s3);
\draw[] (m4) -- node[left]{} (s2);
\draw[] (d2') -- node[left]{} (s3);
\draw[] (m5) -- node[left]{} (s3);
\draw[] (m5) -- node[left]{} (s4);
\draw[] (m6) -- node[left]{} (s4);
\draw[] (m6) -- node[left]{} (s3);
\end{tikzpicture}
\end{matrix}
\end{gather}
All simple $\mathcal{V}{\rm ir}$-modules appearing in this diagram are distinct. The simple modules which only receive arrows are the composition factors of the socle of $\mathcal{K}_{mp-1,nq-1}$, the simple modules which both receive and originate arrows are the composition factors of the middle layer of the socle series, and the simple modules which only originate arrows are the composition factors of the top layer of the socle series. Each arrow in the diagram signifies an indecomposable subquotient of $\mathcal{K}_{mp-1,nq-1}$ of length $2$.

From \eqref{eqn:K-diag}, the simple quotients of $\mathcal{K}_{mp-1,nq-1}$ are precisely the simple quotients claimed for~$\mathcal{L}_{mp-1,1}\boxtimes\mathcal{L}_{np-1,1}$ in the theorem statement. Also, for $m=n$, $\mathcal{K}_{np-1,nq-1}$ has a quotient~$\widetilde{\mathcal{K}}_{1,1}$ of length $2$ with a non-split short exact sequence
\begin{gather*}
0\longrightarrow\mathcal{L}_{1,1}\longrightarrow\widetilde{\mathcal{K}}_{1,1}\longrightarrow\mathcal{L}_{2p-1,1}\longrightarrow 0.
\end{gather*}
Taking the contragredient of this short exact sequence and observing that $h_{1,1}<h_{2p-1,1}$, we~see that $\widetilde{\mathcal{K}}_{1,1}'$ is generated by a highest-weight vector of conformal weight $h_{1,1}=0$ and thus is a~length $2$ quotient of the Virasoro Verma module of lowest conformal weight~$0$. By the same argument applied to~\eqref{eqn:Kac-1}, this length $2$ quotient must also be isomorphic to~$\mathcal{K}_{1,1}$, and it follows that $\widetilde{\mathcal{K}}_{1,1}\cong\mathcal{K}_{1,1}'$. Thus the formulas claimed for $\mathcal{L}_{mp-1,1}\boxtimes\mathcal{L}_{nq-1,1}$ in the theorem statement yield the quotient of $\mathcal{K}_{mp-1,nq-1}/M$, where $M$ is the submodule containing $\mathrm{Soc}(\mathcal{K}_{mp-1,nq-1})$ such that $M/\mathrm{Soc}(\mathcal{K}_{mp-1,nq-1})$ is the direct sum of all simple modules appearing in the middle layer of the socle series of $\mathcal{K}_{mp-1,nq-1}$ except for $\mathcal{L}_{1,1}$ in the $m=n$ case.

We now show that the kernel $K$ of the surjection $\mathcal{K}_{mp-1,nq-1}\twoheadrightarrow\mathcal{L}_{mp-1,1}\boxtimes\mathcal{L}_{np-1,1}$ from
 the right exact sequence \eqref{eqn:right-exact} is contained in $M$, so that we get a surjection
 \begin{gather}\label{eqn:fus-rule-surj}
 \mathcal{L}_{mp-1,1}\boxtimes\mathcal{L}_{np-1,1}\twoheadrightarrow\mathcal{K}_{mp-1,nq-1}/M.
 \end{gather}
 Indeed, by considering the composition
 \begin{gather*}
 \mathcal{K}_{mp+1,nq-1}\oplus\mathcal{K}_{mp-1,nq+1}\twoheadrightarrow K\twoheadrightarrow K/(K\cap M)\hookrightarrow \mathcal{K}_{mp-1,nq-1}/M,
 \end{gather*}
 we see that any simple quotient of $K/(K\cap M)$ is a homomorphic image of either $\mathcal{K}_{mp+1,nq-1}$ or~$\mathcal{K}_{mp-1,nq+1}$ and is also a composition factor of $\mathcal{K}_{mp-1,nq-1}/M$. However, the diagrams in \cite[Section 2.3]{MS-cpq-Vir} show that the simple quotients of $\mathcal{K}_{mp+1,nq-1}$ are
\begin{gather*}
\mathcal{L}_{(m-n+2)p+1,1},\mathcal{L}_{(m-n+4)p+1,1},\ldots,\mathcal{L}_{(m+n-2)p+1,1},
\end{gather*}
and the simple quotients of $\mathcal{K}_{mp-1,nq+1}$ are
\begin{gather}
\mathcal{L}_{1,(m-n)q+1},\mathcal{L}_{1,(m-n+2)q+1},\ldots,\mathcal{L}_{1,(m+n-2)q+1}\qquad \text{if}\quad m>n,\nonumber\\
\mathcal{L}_{1,2q+1},\mathcal{L}_{1,4q+1},\ldots,\mathcal{L}_{1,(2n-2)q+1}\qquad \text{if} \quad m=n.\label{eqn:comp-factors}
\end{gather}
Thus from \eqref{eqn:K-diag}, combining the simple quotients of $\mathcal{K}_{mp+1,nq-1}$ and $\mathcal{K}_{mp-1,nq+1}$ yields all composition factors of the middle socle layer of $\mathcal{K}_{mp-1,nq-1}$, except for $\mathcal{L}_{1,1}$ in the $m=n$ case. This means the simple quotients of $\mathcal{K}_{mp+1,nq-1}$ and $\mathcal{K}_{mp-1,nq+1}$ are disjoint from the composition factors of $\mathcal{K}_{mp-1,nq-1}/M$, showing that $K/(K\cap M)$ has no simple quotients. Thus $K\cap M = K$, that is, $K\subseteq M$. This establishes the surjection \eqref{eqn:fus-rule-surj}.

We still need to show that \eqref{eqn:fus-rule-surj} is an isomorphism, that is, $M=K$. It is enough to show that $K$ contains all composition factors in the middle socle layer of $\mathcal{K}_{mp-1,nq-1}$, except for $\mathcal{L}_{1,1}$ in the $m=n$ case. Again, from \eqref{eqn:K-diag}, the composition factors of the middle socle layer are
\begin{align*}
&\mathcal{L}_{1,iq+1},\qquad m-n\leq i\leq m+n-2,\qquad i+m+n\equiv 0 \mod 2,\\
 &\mathcal{L}_{jp+1,1}, \qquad m-n+2\leq j\leq m+n-2,\qquad j+m+n\equiv 0 \mod 2.
\end{align*}
For any of these simple modules $W=\mathcal{L}_{1,iq+1}$ or $\mathcal{L}_{jp+1,1}$, let $\widetilde{W}\subseteq\mathcal{K}_{mp-1,nq-1}$ be the submodule containing $\mathrm{Soc}(\mathcal{K}_{mp-1,nq-1})$ such that \smash{$\widetilde{W}/\mathrm{Soc}(\mathcal{K}_{mp-1,nq-1})$} is the direct sum of all composition factors of the middle socle layer except for $W$.
Thus if $W$ occurs as a composition factor of~$\mathcal{L}_{mp-1,1}\boxtimes\mathcal{L}_{np-1,1}$, then $K\subseteq\widetilde{W}$, and hence there is a surjection
\begin{gather*}
\mathcal{L}_{mp-1,1}\boxtimes\mathcal{L}_{np-1,1}\twoheadrightarrow\mathcal{K}_{mp-1,nq-1}/\widetilde{W}.
\end{gather*}
We denote the indecomposable summand of $\mathcal{K}_{mp-1,nq-1}/\widetilde{W}$ that contains $W$ as a submodule by~$Q_i$ if $W=\mathcal{L}_{1,iq+1}$ and by $P_j$ if $W=\mathcal{L}_{jp+1,1}$. Thus to complete the proof of the theorem, it is enough to show that there is no surjection $\mathcal{L}_{mp-1,1}\boxtimes\mathcal{L}_{np-1,1}\twoheadrightarrow Q_i$ or $P_j$ for $m-n\leq i\leq m+n-2$, $i+m+n\equiv 0 \mod 2$ or $m-n+2\leq j\leq m+n-2$, $j+m+n\equiv 0 \mod 2$, except in the case that $m=n$ and $i=0$.

We first consider $Q_{m-n}$. By \eqref{eqn:K-diag}, there is a non-split exact sequence
\begin{gather*}
0\longrightarrow\mathcal{L}_{1,(m-n)q+1}\longrightarrow Q_{m-n}\longrightarrow\mathcal{L}_{1,(m-n+2)q-1}\longrightarrow 0.
\end{gather*}
Since $h_{1,(m-n)q+1}<h_{1,(m-n+2)q-1}$, the same argument as used previously in the $m=n$ case, using \eqref{eqn:Kac-1}, implies that \smash{$Q_{m-n}\cong\mathcal{K}_{1,(m-n)q+1}'$}. Thus using symmetries of vertex algebraic intertwining operators from \cite{FHL},
\begin{gather*}
\operatorname{Hom}(\mathcal{L}_{mp-1,1} \boxtimes \mathcal{L}_{np-1,1}, Q_{m-n}) \cong \operatorname{Hom}(\mathcal{L}_{mp-1,1} \boxtimes\mathcal{K}_{1,(m-n)q+1}, \mathcal{L}_{1,nq-1}).
\end{gather*}
So by \eqref{eqn:Kr1xK1s} and \eqref{eqn:Kac-1}, a surjection $\mathcal{L}_{mp-1,1}\boxtimes\mathcal{L}_{np-1,1}\twoheadrightarrow Q_{m-n}$ would induce a surjection
\begin{gather*}
\mathcal{K}_{mp-1,(m-n)q+1}\xrightarrow{\sim}\mathcal{K}_{mp-1,1}\boxtimes\mathcal{K}_{1,(m-n)q+1}\twoheadrightarrow \mathcal{L}_{mp-1,1}\boxtimes\mathcal{K}_{1,(m-n)q+1}\twoheadrightarrow\mathcal{L}_{1,nq-1}.
\end{gather*}
Such a surjection follows from \eqref{eqn:Kac-1} if $m=n$, but if $m>n$, then replacing $n\mapsto m-n$ in \eqref{eqn:comp-factors} shows that the simple quotients of $\mathcal{K}_{mp-1,(m-n)q+1}$ are given by
\begin{gather*}
\mathcal{L}_{1,nq+1},\mathcal{L}_{1,(n+2)q+1},\ldots,\mathcal{L}_{1,(2m-n-2)q+1}.
\end{gather*}
This set of modules does not include $\mathcal{L}_{1,nq-1}$, so $Q_{m-n}$ cannot be a simple quotient of $\mathcal{L}_{mp-1,1}\boxtimes\mathcal{L}_{np-1,1}$ except in the case $m=n$.

We next consider $Q_i$ for $m-n+2\leq i\leq m+n-2$, $i+m+n\equiv 0 \mod 2$. From \eqref{eqn:K-diag}, there is a length $2$ submodule $\widetilde{Q}_i\subseteq Q_i$ with a short exact sequence
\begin{gather*}
0\longrightarrow\mathcal{L}_{1,iq+1}\longrightarrow\widetilde{Q}_i\longrightarrow\mathcal{L}_{1,iq-1}\longrightarrow 0.
\end{gather*}
Comparing with \eqref{eqn:Kac-1}, $\widetilde{Q}_i\cong\mathcal{K}_{1,iq-1}$ since both must be the same length $2$ quotient of the Verma module of lowest conformal weight $h_{1,iq-1}$. Then \eqref{eqn:K-diag} again yields an exact sequence
\begin{gather*}
0\longrightarrow\mathcal{K}_{1,iq-1}\longrightarrow Q_i\longrightarrow\mathcal{L}_{1,(i+2)q-1}\longrightarrow 0,
\end{gather*}
when $m-n+2\leq i< m+n-2$, and $Q_{m+n-2}\cong\mathcal{K}_{1,(m+n-2)q-1}$. Now to show that $Q_i$ is not a~quotient of $\mathcal{L}_{mp-1,1}\boxtimes\mathcal{L}_{np-1,1}$, it is enough to show that there is no surjection
\begin{gather*}
\mathcal{K}_{1,2}\boxtimes\mathcal{L}_{mp-1,1}\boxtimes\mathcal{L}_{np-1,1}\twoheadrightarrow\mathcal{K}_{1,2}\boxtimes Q_i,
\end{gather*}
since $\boxtimes$ is right exact. Using \cite[Theorems 6.5 and 6.8]{MS-cpq-Vir} (see also \eqref{eqn:K12xKrs}), there is an exact sequence
\begin{gather*}
0 \longrightarrow \mathcal{K}_{1, iq-2} \oplus \mathcal{K}_{1, iq} \longrightarrow \mathcal{K}_{1, 2} \boxtimes Q_i \longrightarrow \mathcal{L}_{1, (i+2)q-2} \longrightarrow 0
\end{gather*}
for $m-n+2\leq i<m+n-2$, while
\begin{gather*}
\mathcal{K}_{1,2}\boxtimes Q_{m+n-2}\cong\mathcal{K}_{1,(m+n-2)q-2}\oplus\mathcal{K}_{1,(m+n-2)q}.
\end{gather*}
In either case, $\mathcal{K}_{1,iq}$ is a direct summand of $\mathcal{K}_{1,2}\boxtimes Q_i$ since $h_{1,iq}\neq h_{1,iq-2},h_{1,(i+2)q-2} \mod \mathbb{Z}$. Thus a surjection $\mathcal{L}_{1,mp-1}\boxtimes\mathcal{L}_{1,np-1}\twoheadrightarrow Q_i$ would induce a surjection
\begin{align*}
\mathcal{K}_{mp-1,nq-2}& \xrightarrow{\sim}\mathcal{K}_{mp-1,1}\boxtimes\mathcal{K}_{1,nq-2} \twoheadrightarrow\mathcal{L}_{mp-1,1}\boxtimes\mathcal{L}_{1,nq-2} \nonumber\\
& \xrightarrow{\sim}\mathcal{K}_{1,2}\boxtimes\mathcal{L}_{mp-1,1}\boxtimes\mathcal{L}_{1,nq-1}\twoheadrightarrow\mathcal{K}_{1,2}\boxtimes Q_i\twoheadrightarrow\mathcal{K}_{1,iq}.
\end{align*}
From the diagrams in \cite[Section 2.3]{MS-cpq-Vir}, it is clear that $\mathcal{K}_{1,iq}\cong\mathcal{L}_{1,iq}$, but that $\mathcal{K}_{mp-1,nq-2}$ has no simple quotient of the form $\mathcal{L}_{1,iq}$ if $q\nmid (nq-2)$. For the remaining $q=2$ case, $\mathcal{K}_{mp-1,nq-2}=\mathcal{K}_{mp-1,(n-1)q}$ has simple quotients
$
\mathcal{L}_{1,(m-n+1)q},\mathcal{L}_{1,(m-n+3)q},\ldots,\mathcal{L}_{1,(m+n-3)q}$,
and these do not include $\mathcal{L}_{1,iq}$ since $i+m+n\equiv 0 \mod 2$. This proves that $Q_i$ is not a quotient of $\mathcal{L}_{mp-1,1}\boxtimes\mathcal{L}_{np-1,1}$, and therefore $\mathcal{L}_{1,iq+1}$ is not a composition factor of $\mathcal{L}_{mp-1,1}\boxtimes\mathcal{L}_{np-1,1}$.

Finally, the case of $P_j$ for $m-n+2\leq j\leq m+n-2$, $j+m+n\equiv 0 \mod 2$ is similar by~${c_{p,q}=c_{q,p}}$ symmetry. In this case, \eqref{eqn:K-diag} and \eqref{eqn:Kac-1} yield a short exact sequence
\begin{gather*}
0\longrightarrow\mathcal{K}_{jp-1,1}\longrightarrow P_j\longrightarrow\mathcal{L}_{(j+2)p-1,1}\longrightarrow 0
\end{gather*}
when $m-n+2\leq j < m+n-2$, and $P_{m+n-2}\cong\mathcal{K}_{(m+n-2)p-1,1}$. Then similar to the $Q_i$ case, a~surjection $\mathcal{L}_{mp-1,1}\boxtimes\mathcal{L}_{np-1,1}\twoheadrightarrow P_j$ would induce a surjection
\begin{align*}
\mathcal{K}_{mp-2,nq-1} &\xrightarrow{\sim}\mathcal{K}_{mp-2,1}\boxtimes\mathcal{K}_{1,nq-1} \twoheadrightarrow\mathcal{L}_{mp-2,1}\boxtimes\mathcal{L}_{1,nq-1} \nonumber\\
& \xrightarrow{\sim}\mathcal{K}_{2,1}\boxtimes\mathcal{L}_{mp-1,1}\boxtimes\mathcal{L}_{np-1,1}\twoheadrightarrow\mathcal{K}_{2,1}\boxtimes P_j\twoheadrightarrow\mathcal{K}_{jp,1}.
\end{align*}
But $\mathcal{K}_{jp,1}\cong\mathcal{L}_{jp,1}$, and the diagrams in \cite[Section 2.3]{MS-cpq-Vir} show that $\mathcal{L}_{jp,1}$ is not a simple quotient of $\mathcal{K}_{mp-2,nq-1}$ (whether $p=2$ or $p>2$). Thus there is no surjection $\mathcal{L}_{mp-1,1}\boxtimes\mathcal{L}_{np-1,1}\twoheadrightarrow P_j$, and therefore $\mathcal{L}_{jp+1,1}$ is not a composition factor of $\mathcal{L}_{mp-1,1}\boxtimes\mathcal{L}_{np-1,1}$ for $m-n+2\leq j\leq m+n-2$, $j+m+n\equiv 0 \mod 2$. This completes the proof of the theorem.
\end{proof}

\section[An sl\_2-type tensor category]{An $\boldsymbol{\mathfrak{sl}_2}$-type tensor category}\label{sec:sl2-type-cat}

The $\mathfrak{sl}_2$-like fusion rules of Theorem~\ref{thm:Vir-fus-rules} suggest that we consider the additive full subcategory~${\mathcal{C}_{\mathfrak{sl}_2}\subseteq\mathcal{O}_{c_{p,q}}}$ such that
\begin{gather}\label{eqn:indecom-in-Csl2}
\lbrace\mathcal{K}_{1,1}'\rbrace\cup\lbrace\mathcal{L}_{np-1,1}\mid n\in\mathbb{Z}_{\geq 3}\rbrace
\end{gather}
is a complete list of representatives of the isomorphism classes of indecomposable objects in~$\mathcal{C}_{\mathfrak{sl}_2}$. That is, every object of $\mathcal{C}_{\mathfrak{sl}_2}$ is isomorphic to a finite direct sum of objects from \eqref{eqn:indecom-in-Csl2}. In~particular, $\mathcal{K}_{1,1}$ and $\mathcal{L}_{2p-1,1}$ are not objects of $\mathcal{C}_{\mathfrak{sl}_2}$. In this section, we will show that $\mathcal{C}_{\mathfrak{sl}_2}$ is a tensor subcategory of $\mathcal{O}_{c_{p,q}}$ with different unit object $\mathcal{K}_{1,1}'$, and that $\mathcal{C}_{\mathfrak{sl}_2}$ is rigid and tensor equivalent to some $3$-cocycle twist of the category $\operatorname{Rep} \mathfrak{sl}_2$ of finite-dimensional $\mathfrak{sl}_2$-modules.

In view of Theorem~\ref{thm:Vir-fus-rules}, to show that $\mathcal{C}_{\mathfrak{sl}_2}$ is closed under the tensor product $\boxtimes$ on $\mathcal{O}_{c_{p,q}}$, it only remains to show that if $\mathcal{W}$ is an object of $\mathcal{C}_{\mathfrak{sl}_2}$, then $\mathcal{K}_{1,1}'\boxtimes W$ is also an object of $\mathcal{C}_{\mathfrak{sl}_2}$. In~fact, since we want $\mathcal{K}_{1,1}'$ to be the unit object of $\mathcal{C}_{\mathfrak{sl}_2}$, we want $\mathcal{K}_{1,1}'\boxtimes W\cong W$. To prove this, we~need the following lemma (compare with \cite[Lemma~2.19]{MY-univ-sl2}).
\begin{Lemma}\label{lem:L11xW}
If $W$ is an object of $\mathcal{C}_{\mathfrak{sl}_2}$, then $\mathcal{L}_{1,1}\boxtimes W=0$.
\end{Lemma}
\begin{proof}
The simple VOA quotient $L_{c_{p,q}}$ of $V_{c_{p,q}}$ is isomorphic to $\mathcal{L}_{1,1}$ as a $\mathcal{V}{\rm ir}$-module. Also, $L_{c_{p,q}}$ is a rational VOA, and every simple $L_{c_{p,q}}$-module is isomorphic to $\mathcal{L}_{r,s}$ for some $1\leq r\leq p-1$ and $1\leq s\leq q-1$ \cite{Wang}.
If $W$ is any object of $\mathcal{O}_{c_{p,q}}$, then $\mathcal{L}_{1,1}\boxtimes W$ is an $L_{c_{p,q}}$-module by \cite[Lemma 5.11]{MS-cpq-Vir}. On the other hand, the quotient map $\mathcal{K}_{1,1}\twoheadrightarrow\mathcal{L}_{1,1}$ induces a surjection
\begin{gather*}
W\xrightarrow{\sim}\mathcal{K}_{1,1}\boxtimes W\twoheadrightarrow\mathcal{L}_{1,1}\boxtimes W.
\end{gather*}
Thus $\mathcal{L}_{1,1}\boxtimes W$ is an $L_{c_{p,q}}$-module which is a quotient of $W$. Then if $W$ is an object of $\mathcal{C}_{\mathfrak{sl}_2}$ the only such quotient of $W$ is $0$, since any simple quotient of $W$ is isomorphic to $\mathcal{L}_{np-1,1}$ for some~${n\geq 2}$ (note from \eqref{eqn:Kac-1} that $\mathcal{L}_{2p-1,1}$ is the unique simple quotient of $\mathcal{K}_{1,1}'$).
\end{proof}

We now define left and right unit isomorphism candidates in $\mathcal{C}_{\mathfrak{sl}_2}$.
First, it follows from~\eqref{eqn:Kac-1}, the exactness of the contragredient functor, and the fact that simple $\mathcal{V}{\rm ir}$-modules are self-contragredient, that we have exact sequences
\begin{align}
& 0\longrightarrow\mathcal{L}_{2p-1,1}\longrightarrow\mathcal{K}_{1,1}\xrightarrow{ \pi }\mathcal{L}_{1,1}\longrightarrow 0,\nonumber\\
& 0\longrightarrow\mathcal{L}_{1,1}\xrightarrow{ \eta }\mathcal{K}_{1,1}'\longrightarrow\mathcal{L}_{2p-1,1}\longrightarrow 0. \label{eqn:K11-K11'-exact-seq}
\end{align}
We fix $\varphi=\eta\circ\pi\colon \mathcal{K}_{1,1}'\rightarrow\mathcal{K}_{1,1}$, and then for any object $W$ in $\mathcal{C}_{\mathfrak{sl}_2}$, we define homomorphisms
\begin{gather}
l_W'\colon \ \mathcal{K}_{1,1}'\boxtimes W\xrightarrow{\varphi\boxtimes\mathrm{Id}_W} \mathcal{K}_{1,1}\boxtimes W\xrightarrow{l_W} W,\nonumber\\
 r_W'\colon \ W\boxtimes\mathcal{K}_{1,1}' \xrightarrow{\mathrm{Id}_W\boxtimes\varphi} W\boxtimes\mathcal{K}_{1,1}\xrightarrow{r_W} W,\label{eqn:ln'-rn'-def}
\end{gather}
where $l_W$ and $r_W$ are the left and right unit isomorphisms in $\mathcal{O}_{c_{p,q}}$, respectively.

\begin{Proposition}\label{prop:K11'xW}
For any object $W$ in $\mathcal{C}_{\mathfrak{sl}_2}$, $l_W'$ and $r_W'$ are isomorphisms.
\end{Proposition}
\begin{proof}
By elementary properties of braided tensor categories, $r_W=l_W\circ\mathcal{R}_{W,\mathcal{K}_{1,1}}$, where $\mathcal{R}$ is the natural braiding isomorphism in \smash{$\mathcal{O}_{c_{p,q}}$}. Thus naturality implies \smash{$r_W'=l_W'\circ\mathcal{R}_{W,\mathcal{K}_{1,1}'}$}, and hence $l_W'$ is an isomorphism if and only if $r_W'$ is.

To show $l_W'$ is an isomorphism, it is enough to show that $\varphi\boxtimes\mathrm{Id}_W$ is an isomorphism, and for this it is enough to show that
\begin{gather*}
\begin{split}
& \pi\boxtimes\mathrm{Id}_W\colon \ \mathcal{K}_{1,1}'\boxtimes W\longrightarrow\mathcal{L}_{2p-1,1}\boxtimes W,\\
& \eta\boxtimes\mathrm{Id}_W\colon \ \mathcal{L}_{2p-1,1}\boxtimes W\longrightarrow\mathcal{K}_{1,1}\boxtimes W
\end{split}
\end{gather*}
are both isomorphisms. By \eqref{eqn:K11-K11'-exact-seq} and right exactness of $\boxtimes$, we have right exact sequences
\begin{align*}
&\mathcal{L}_{1,1}\boxtimes W\longrightarrow\mathcal{K}_{1,1}'\boxtimes W\xrightarrow{\pi\boxtimes\mathrm{Id}_W} \mathcal{L}_{2p-1,1}\boxtimes W\longrightarrow 0,\\
& \mathcal{L}_{2p-1,1}\boxtimes W\xrightarrow{\eta\boxtimes\mathrm{Id}_W} \mathcal{K}_{1,1}\boxtimes W\longrightarrow\mathcal{L}_{1,1}\boxtimes W\longrightarrow 0.
\end{align*}
It then follows from Lemma \ref{lem:L11xW} that $\pi\boxtimes\mathrm{Id}_W$ is an isomorphism and $\eta\boxtimes\mathrm{Id}_W$ is surjective.

To show that $\eta\boxtimes\mathrm{Id}_W$ is also injective and thus an isomorphism, it suffices to show that $\mathcal{L}_{2p-1,1}\boxtimes W\cong W$, since then the finite-dimensional weight spaces of $\mathcal{L}_{2p-1,1}\boxtimes W$ and $\mathcal{K}_{1,1}\boxtimes W\cong W$ will have the same dimension. To prove $\mathcal{L}_{2p-1,1}\boxtimes W\cong W$, we may assume $W$ is indecomposable, and the case $W=\mathcal{L}_{np-1,1}$ for $n\geq 3$ is covered by Theorem~\ref{thm:Vir-fus-rules}. For the case~${W=\mathcal{K}_{1,1}'}$, \eqref{eqn:K11-K11'-exact-seq} and right exactness of $\boxtimes$ yield a right exact sequence
\begin{gather*}
\mathcal{L}_{2p-1,1}\boxtimes\mathcal{L}_{1,1}\longrightarrow\mathcal{L}_{2p-1,1}\boxtimes\mathcal{K}_{1,1}'\longrightarrow\mathcal{L}_{2p-1,1}\boxtimes\mathcal{L}_{2p-1,1}\longrightarrow 0.
\end{gather*}
Since $\mathcal{L}_{2p-1,1}\boxtimes\mathcal{L}_{1,1}=0$ by the same argument as in the proof of Lemma \ref{lem:L11xW}, and since $\mathcal{L}_{2p-1,1}\boxtimes\mathcal{L}_{2p-1,1}\cong\mathcal{K}_{1,1}'$ by Theorem~\ref{thm:Vir-fus-rules}, we get $\mathcal{L}_{2p-1,1}\boxtimes\mathcal{K}_{1,1}'\cong\mathcal{K}_{1,1}'$, as required.
\end{proof}

Theorem~\ref{thm:Vir-fus-rules} and Proposition \ref{prop:K11'xW} now show in particular that $\mathcal{C}_{\mathfrak{sl}_2}$ is closed under the tensor product on $\mathcal{O}_{c_{p,q}}$.
To write the tensor products of indecomposable objects in $\mathcal{C}_{\mathfrak{sl}_2}$ in a uniform way, we introduce the notation
\begin{gather}\label{eqn:Ln-notation}
\mathcal{L}_n =\begin{cases}
\mathcal{K}_{1,1}' & \text{if } n=0,\\
\mathcal{L}_{(n+2)p-1,1} & \text{if } n\geq 1.\\
\end{cases}
\end{gather}
for $n\in\mathbb{Z}_{\geq 0}$. Then Theorem~\ref{thm:Vir-fus-rules} and Proposition \ref{prop:K11'xW} imply
\begin{gather}\label{eqn:unif-fus-rules}
\mathcal{L}_m\boxtimes\mathcal{L}_n\cong\bigoplus_{\substack{k=\vert m-n\vert\\
k+m+n\equiv 0 \mod 2}}^{m+n} \mathcal{L}_k
\end{gather}
for all $m,n\in\mathbb{Z}_{\geq 0}$. These are precisely the fusion rules of finite-dimensional simple $\mathfrak{sl}_2$-modules, if we identify $\mathcal{L}_n$ with the $(n+1)$-dimensional simple $\mathfrak{sl}_2$-module of highest weight $n$.

Proposition \ref{prop:K11'xW} also shows that there are natural isomorphisms $l'\colon\mathcal{K}_{1,1}'\boxtimes\bullet\longrightarrow\mathrm{Id}_{\mathcal{C}_{\mathfrak{sl}_2}}$, $r'\colon\bullet\boxtimes\mathcal{K}_{1,1}'\longrightarrow\mathrm{Id}_{\mathcal{C}_{\mathfrak{sl}_2}}$. Further, the associativity isomorphisms $\mathcal{A}$ and braiding isomorphisms $\mathcal{R}$ on~$\mathcal{O}_{c_{p,q}}$ restrict to natural isomorphisms on $\mathcal{C}_{\mathfrak{sl}_2}$ that satisfy the pentagon and hexagon axioms of a braided monoidal category.

\begin{Theorem}\label{thm:Csl2-tensor}
$(\mathcal{C}_{\mathfrak{sl}_2},\boxtimes,\mathcal{K}_{1,1}',\mathcal{A},l',r',\mathcal{R})$ is a semisimple braided tensor category.
\end{Theorem}
\begin{proof}
To show that $\mathcal{C}_{\mathfrak{sl}_2}$ is semisimple, it is enough to show that $\mathcal{K}_{1,1}'$ is simple as an object of~$\mathcal{C}_{\mathfrak{sl}_2}$ (though it is not simple in $\mathcal{O}_{c_{p,q}}$). Indeed, from \eqref{eqn:K11-K11'-exact-seq}, the only non-trivial subobject of~$\mathcal{K}_{1,1}'$ is $\mathcal{L}_{1,1}$, which is not an object of $\mathcal{C}_{\mathfrak{sl}_2}$. Thus $\mathcal{C}_{\mathfrak{sl}_2}$ is a semisimple abelian category.

To show that $\mathcal{C}_{\mathfrak{sl}_2}$ is a braided tensor category, it remains to prove the triangle axiom for~$l'$,~$r'$, and $\mathcal{A}$. Indeed, since the associativity isomorphisms are natural and the triangle axiom holds in $\mathcal{O}_{c_{p,q}}$, we have
\begin{align*}
(r'_{W_1}\boxtimes\mathrm{Id}_{W_2})\circ\mathcal{A}_{W_1,\mathcal{K}_{1,1}',W_2} & = (r_{W_1}\boxtimes\mathrm{Id}_{W_2})\circ((\mathrm{Id}_{W_1}\boxtimes\varphi)\boxtimes\mathrm{Id}_{W_2})\circ\mathcal{A}_{W_1,\mathcal{K}_{1,1}',W_2}\nonumber\\
& = (r_{W_1}\boxtimes\mathrm{Id}_{W_2})\circ\mathcal{A}_{W_1,\mathcal{K}_{1,1},W_2}\circ(\mathrm{Id}_{W_1}\boxtimes(\varphi\boxtimes\mathrm{Id}_{W_2}))\nonumber\\
& =(\mathrm{Id}_{W_1}\boxtimes l_{W_2})\circ(\mathrm{Id}_{W_1}\boxtimes(\varphi\boxtimes\mathrm{Id}_{W_2})) = \mathrm{Id}_{W_1}\boxtimes l_{W_2}'
\end{align*}
for any objects $W_1$, $W_2$ in $\mathcal{C}_{\mathfrak{sl}_2}$, as required.
\end{proof}

We next show that the tensor category $\mathcal{C}_{\mathfrak{sl}_2}$ is rigid. Actually, this follows rather easily using Theorem 1.1 of the recent paper \cite{EP}, which appeared after we had already begun this work. Here we present a more explicit and self-contained proof, starting by showing that $\mathcal{L}_1$ is rigid and self-dual. The key point is to determine the associativity isomorphism $\mathcal{A}_{\mathcal{L}_1,\mathcal{L}_1,\mathcal{L}_1}\colon \mathcal{L}_1\boxtimes(\mathcal{L}_1\boxtimes\mathcal{L}_1)\rightarrow(\mathcal{L}_1\boxtimes\mathcal{L}_1)\boxtimes\mathcal{L}_1$ explicitly enough. Since $\mathcal{C}_{\mathfrak{sl}_2}$ is semisimple, this associativity isomorphism is determined by the $F$-matrix, or $6j$-symbols, and it turns out that proving rigidity amounts to showing that a certain $F$-matrix entry is non-zero. We will prove this by using the hexagon axiom of a braided tensor category to constrain the $F$-matrix. To prepare, we fix some notation.

Recalling the notation \eqref{eqn:Ln-notation} and the fusion rules \eqref{eqn:unif-fus-rules}, let
\begin{gather*}
\pi^k_{mn}\colon\ \mathcal{L}_m\boxtimes\mathcal{L}_n\longrightarrow\mathcal{L}_k,\qquad i_{mn}^k\colon\ \mathcal{L}_k\longrightarrow\mathcal{L}_m\boxtimes\mathcal{L}_n
\end{gather*}
for $m,n\in\mathbb{Z}_{\geq 0}$, $\vert m-n\vert\leq k\leq m+n$, $k+m+n\equiv 0 \mod 2$ denote a system of surjections and injections such that
\begin{gather}\label{eqn:proj-system}
\pi^k_{mn}\circ i^{k'}_{mn} =\delta_{k,k'}\mathrm{Id}_{\mathcal{L}_k},\qquad \sum_{\substack{k=\vert m-n\vert\\ k+m+n\equiv 0 \mod 2}}^{m+n} i_{mn}^k\circ\pi_{mn}^k =\mathrm{Id}_{\mathcal{L}_m\boxtimes\mathcal{L}_n}.
\end{gather}
For simplicity, we may assume that if $m<n$, then $\pi^k_{mn}\circ\mathcal{R}_{\mathcal{L}_m,\mathcal{L}_n}=\pi^k_{nm}$. We may also assume that $\pi_{0n}^n = l'_{\mathcal{L}_n}$ and $\pi_{n0}^n=r'_{\mathcal{L}_n}$ for all $n\in\mathbb{Z}_{\geq 0}$. For the case $m=n$, we have
\begin{gather*}
\pi_{nn}^k\circ\mathcal{R}_{\mathcal{L}_n,\mathcal{L}_n} = R_n^k\cdot\pi_{nn}^k
\end{gather*}
for some $R_n^k\in\mathbb{C}^\times$. In fact, using the methods of \cite[Sections 7 and 8]{GaNe} and \cite[Section 6]{MY-25-Vir} (see in particular the calculations preceding \cite[Theorem 6.3]{MY-25-Vir}), we get
\begin{gather*}
R_n^k ={\rm e}^{\pi {\rm i}(h_{(k+2)p-1,1}-2h_{(n+2)p-1,1})}.
\end{gather*}
We will need the case of $n=1$, in particular,
\begin{gather}\label{eqn:r-values}
R_1^0 = {\rm e}^{\pi {\rm i} pq/2},\qquad R_1^2 = - {\rm e}^{\pi {\rm i} pq/2}.
\end{gather}

Next we consider the triple tensor product $\mathcal{L}_1\boxtimes(\mathcal{L}_1\boxtimes\mathcal{L}_1)$. For $k=0,2$, we define morphisms
\begin{gather*}
 \Pi_k\colon\ \mathcal{L}_1\boxtimes(\mathcal{L}_1\boxtimes\mathcal{L}_1)\xrightarrow{\mathrm{Id}_{\mathcal{L}_1}\boxtimes\pi_{11}^k} \mathcal{L}_1\boxtimes\mathcal{L}_k\xrightarrow{\pi_{1k}^1} \mathcal{L}_1,\\
 \widetilde{\Pi}_k\colon\ (\mathcal{L}_1\boxtimes\mathcal{L}_1)\boxtimes\mathcal{L}_1\xrightarrow{\pi_{11}^k\boxtimes\mathrm{Id}_{\mathcal{L}_1}} \mathcal{L}_1\boxtimes\mathcal{L}_k\xrightarrow{\pi_{k1}^1} \mathcal{L}_1.
\end{gather*}
It is easy to see that $\lbrace\Pi_0,\Pi_2\rbrace$, respectively $\lbrace\widetilde{\Pi}_0,\widetilde{\Pi}_2\rbrace$ is a basis of $\operatorname{Hom}(\mathcal{L}_1\boxtimes(\mathcal{L}_1\boxtimes\mathcal{L}_1),\mathcal{L}_1)$, respectively $\operatorname{Hom}((\mathcal{L}_1\boxtimes\mathcal{L}_1)\boxtimes\mathcal{L}_1,\mathcal{L}_1)$. Thus we can define the $F$-matrix, or $6j$-symbols, by
\begin{gather*}
\widetilde{\Pi}_k\circ\mathcal{A}_{\mathcal{L}_1,\mathcal{L}_1,\mathcal{L}_1} =\sum_{l=0,2} F_{kl}\Pi_l.
\end{gather*}
Note that the $2\times 2$ matrix \smash{$F=\bigl[\begin{smallmatrix}
F_{00} & F_{02}\\
F_{20} & F_{22}\\
\end{smallmatrix}\bigr]$} is invertible.

\begin{Theorem}\label{thm:Csl2-rigid}
The tensor category $\mathcal{C}_{\mathfrak{sl}_2}$ is rigid.
\end{Theorem}
\begin{proof}
Since $\mathcal{C}_{\mathfrak{sl}_2}$ is braided and semisimple, it is enough to show that each simple object $W$ has a left dual $W^*$. We will first show that $\mathcal{L}_1=\mathcal{L}_{3p-1,1}$ is rigid and self-dual using the evaluation and coevaluation candidates
\begin{gather*}
\pi^0_{11}\colon\ \mathcal{L}_1\boxtimes\mathcal{L}_1\longrightarrow\mathcal{L}_0=\mathcal{K}_{1,1}',\qquad i^0_{11}\colon\ \mathcal{L}_0\longrightarrow\mathcal{L}_1\boxtimes\mathcal{L}_1.
\end{gather*}
Similar to \cite[Lemma 4.2.1]{CMY3} and \cite[Lemma 2.1]{EP}, it is enough to show that the composition
\begin{gather*}
\mathcal{L}_1\xrightarrow{\sim} \mathcal{L}_1\boxtimes\mathcal{L}_0\xrightarrow{\mathrm{Id}_{\mathcal{L}_1}\boxtimes i^0_{11}} \mathcal{L}_1\boxtimes(\mathcal{L}_1\boxtimes\mathcal{L}_1)\xrightarrow{\sim}(\mathcal{L}_1\boxtimes\mathcal{L}_1)\boxtimes\mathcal{L}_1\xrightarrow{\pi^0_{11}\boxtimes\mathrm{Id}_{\mathcal{L}_1}}\mathcal{L}_0\boxtimes\mathcal{L}_1\xrightarrow{\sim}\mathcal{L}_1
\end{gather*}
is a non-zero scalar multiple of $\mathrm{Id}_{\mathcal{L}_1}$. From the definitions and \eqref{eqn:proj-system}, this composition equals
\begin{gather*}
\widetilde{\Pi}_0\circ\mathcal{A}_{\mathcal{L}_1,\mathcal{L}_1,\mathcal{L}_1} \circ\bigl(\mathrm{Id}_{\mathcal{L}_1}\boxtimes i^0_{11}\bigr)\circ (r_{\mathcal{L}_1}')^{-1} =\sum_{k=0,2} F_{0k}\cdot\Pi_k\circ\bigl(\mathrm{Id}_{\mathcal{L}_1}\boxtimes i^0_{11}\bigr)\circ i^1_{10}\\
\qquad =\sum_{k=0,2} F_{0k}\cdot\pi^1_{1k}\circ\bigl(\mathrm{Id}\boxtimes\pi^k_{11}\bigr)\circ\bigl(\mathrm{Id}_{\mathcal{L}_1}\boxtimes i^0_{11}\bigr)\circ i^1_{10} =F_{00}\cdot\mathrm{Id}_{\mathcal{L}_1}.
\end{gather*}
Thus if $F_{00}\neq 0$, then $\mathcal{L}_1$ is rigid with evaluation $\frac{1}{F_{00}}\cdot\pi^0_{11}$ and coevaluation $i^0_{11}$.

To show that $F_{00}\neq 0$, we will use the hexagon axiom, which asserts in particular that the two compositions
\begin{align*}
& \mathcal{L}_1\boxtimes(\mathcal{L}_1\boxtimes\mathcal{L}_1)\xrightarrow{\mathrm{Id}\boxtimes\mathcal{R}} \mathcal{L}_1\boxtimes(\mathcal{L}_1\boxtimes\mathcal{L}_1)\xrightarrow{\mathcal{A}} (\mathcal{L}_1\boxtimes\mathcal{L}_1)\boxtimes\mathcal{L}_1\xrightarrow{\mathcal{R}\boxtimes\mathrm{Id}} (\mathcal{L}_1\boxtimes\mathcal{L}_1)\boxtimes\mathcal{L}_1,\\
& \mathcal{L}_1\boxtimes(\mathcal{L}_1\boxtimes\mathcal{L}_1)\xrightarrow{\mathcal{A}} (\mathcal{L}_1\boxtimes\mathcal{L}_1)\boxtimes\mathcal{L}_1\xrightarrow{\mathcal{R}} \mathcal{L}_1\boxtimes(\mathcal{L}_1\boxtimes{L}_1)\xrightarrow{\mathcal{A}} (\mathcal{L}_1\boxtimes\mathcal{L}_1)\boxtimes\mathcal{L}_1
\end{align*}
are equal; here we have dropped labels on morphisms for brevity. Composing these two compositions with $\widetilde{\Pi}_k$ for $m=0,2$ yields constraints on the $F$-matrix (see, for example, \cite[equation~2.3]{Aboum}), namely
\begin{gather*}
R_1^k F_{kl} R_1^m = \sum_{m=0,2} F_{km} R_{m1}^1 F_{ml}
\end{gather*}
for $k,l\in\lbrace 0,2\rbrace$,
where $R_{m1}^1$ is defined by $\pi_{1m}^1\circ\mathcal{R}_{\mathcal{L}_m,\mathcal{L}_1}=R_{m1}^1\cdot\pi_{m1}^1$. By our conventions, $R_{m1}^1=1$ for $m=0,2$, and thus using~\eqref{eqn:r-values}, the above equation is equivalent to
\begin{gather}\label{eqn:hex-contraint}
(-1)^{pq}\left[\begin{matrix}
F_{00} & -F_{02}\\
-F_{20} & F_{22}\\
\end{matrix}\right] = \left[\begin{matrix}
F_{00} & F_{02}\\
F_{20} & F_{22}\\
\end{matrix}\right]^2.
\end{gather}
The equation for the upper left entry yields
$
(-1)^{pq}F_{00} = F_{00}^2 + F_{02}F_{20}$,
and thus the determinant of $F$ is a multiple of $F_{00}$:
\begin{align*}
\det(F) = F_{00}F_{22}-F_{02}F_{20} = F_{00}F_{22}-(-1)^{pq}F_{00}+F_{00}^2 = F_{00}(F_{22}-(-1)^{pq}+F_{00}).
\end{align*}
Since the matrix $F$ is invertible and thus $\det(F)\neq 0$, it follows that $F_{00}\neq 0$ as well. This proves that $\mathcal{L}_1$ is rigid and self-dual.

Rigidity of the remaining simple objects $\mathcal{L}_n$ for $n\geq 2$ now follows by induction on $n$. Indeed, assuming by induction that $\mathcal{L}_n$ is rigid, then $\mathcal{L}_{n+1}$ is a direct summand of
$
\mathcal{L}_{1}\boxtimes\mathcal{L}_n\cong\mathcal{L}_{n-1}\oplus\mathcal{L}_{n+1}$,
which is rigid and self-dual because it is a tensor product of rigid and self-dual objects of $\mathcal{C}_{\mathfrak{sl}_2}$
(see, for example, \cite[Lemma~A.3]{KL4}).
Thus $\mathcal{L}_{n+1}$ is also rigid and self-dual (see, for example, \cite[Lemma~5.9]{MS-cpq-Vir}.
\end{proof}

\begin{Remark}
We emphasize that the modules $\mathcal{L}_n$, $n\geq 0$, are \textit{not} rigid when considered as objects of $\mathcal{O}_{c_{p,q}}$. This is because the unit object of $\mathcal{O}_{c_{p,q}}$ is different from that of $\mathcal{C}_{\mathfrak{sl}_2}$.
\end{Remark}

\begin{Remark}
It is not difficult to find all solutions of \eqref{eqn:hex-contraint}. Considering the possibilities $F_{02}=0$ and $F_{02}\neq 0$ separately, the invertible solutions for $F$ are given by
\begin{gather*}
\left[\begin{matrix}
(-1)^{pq} & 0\\
0 & (-1)^{pq}\\
\end{matrix}\right],\qquad\quad\left[\begin{matrix}
-\frac{1}{2}(-1)^{pq} & t\\
-\frac{3}{4t} & -\frac{1}{2}(-1)^{pq}\\
\end{matrix}\right],\quad t\in\mathbb{C}^\times.
\end{gather*}
These solutions imply that the intrinsic dimension of $\mathcal{L}_1$, which is defined to be the endomorphism of $\mathcal{L}_0$ obtained by composing the evaluation and coevaluation morphisms of $\mathcal{L}_1$, is either~$\pm 1$ or $\pm 2$. The next theorem will rule out the first possibility, and thus the actual $F$-matrix is given by the second matrix above for some $t\in\mathbb{C}^\times$ (and $t$ will depend on the choice of normalizations for $\pi_{11}^k$, $k=0,2$).
\end{Remark}

\begin{Remark}
As far as we are aware, our proof of Theorem~\ref{thm:Csl2-rigid} is the first rigidity proof for VOAs that uses explicit calculation of $F$-matrices via general categorical principles. In previous rigidity proofs for VOAs, such as \cite{ CMY2, Hu-rigidity, MS-cpq-Vir, MY-cp1-Vir, TW1}, $F$-matrix entries have been calculated or constrained by analytic methods, such as by solving regular singular point differential equations.

In principal, it might be possible to prove $\mathcal{L}_1=\mathcal{L}_{3p-1,1}$ is rigid in $\mathcal{C}_{\mathfrak{sl}_2}$ by such analytic methods, using BPZ partial differential equations derived from explicit expressions for singular vectors given by the Benoit--Saint-Aubin formula~\cite{BS-A}. But although these differential equations are explicit \cite[Section~5.3]{KK}, they have rather high order and do not seem particularly easy to solve explicitly.
Thus in Theorem~\ref{thm:Csl2-rigid} we have used the hexagon axiom to constrain $F$ instead.

Alternatively, as we remarked above, we could use \cite[Theorem 1.1]{EP} to prove $\mathcal{L}_n$ is rigid in $\mathcal{C}_{\mathfrak{sl}_2}$. This would require showing that $\dim\mathrm{End}(\mathcal{L}_n^{\boxtimes m})< m!$ for some $m\in\mathbb{Z}_{\geq 1}$, and this is rather easy from the fusion rules \eqref{eqn:unif-fus-rules}. In examples of vertex algebraic tensor categories where the fusion rules are not fully known, the methods from our proof of Theorem~\ref{thm:Csl2-rigid} may be more useful.
\end{Remark}

By \eqref{eqn:unif-fus-rules} and Theorems \ref{thm:Csl2-tensor} and \ref{thm:Csl2-rigid}, $\mathcal{C}_{\mathfrak{sl}_2}$ is a rigid semisimple tensor category with the same fusion rules as the category $\operatorname{Rep} \mathfrak{sl}_2$ of finite-dimensional $\mathfrak{sl}_2$-modules. Such categories were classified up to tensor equivalence in \cite{KaWe}, so we could use this classification to identify~$\mathcal{C}_{\mathfrak{sl}_2}$.
However, here we will mainly focus on the full tensor subcategory $\mathcal{C}_{\mathrm{PSL}_2}\subseteq\mathcal{C}_{\mathfrak{sl}_2}$ whose objects are isomorphic to finite direct sums of the modules $\mathcal{L}_{2n}$, $n\in\mathbb{Z}_{\geq 0}$. This subcategory has the same fusion rules as the category $\operatorname{Rep} \mathrm{PSL}_2$ of finite-dimensional modules for the algebraic group~$\mathrm{PSL}_2(\mathbb{C})$.
\begin{Theorem}\label{thm:braid-equiv}
The category $\mathcal{C}_{\mathrm{PSL}_2}$ is braided tensor equivalent to $\operatorname{Rep} \mathrm{PSL}_2$.
\end{Theorem}
\begin{proof}
In view of \eqref{eqn:unif-fus-rules} and Theorems \ref{thm:Csl2-tensor} and \ref{thm:Csl2-rigid}, it follows from \cite[Theorem $A_\infty$]{KaWe} that $\mathcal{C}_{\mathfrak{sl}_2}$ is tensor equivalent to some $3$-cocycle twist of the category $\operatorname{Rep} U_\zeta(\mathfrak{sl}_2)$ of finite-dimensional modules for the quantum group of $\mathfrak{sl}_2$ at $\zeta=\pm 1$ or $\zeta$ not a root of unity. Here $\zeta$ denotes a~square root of the parameter denoted $q$ in \cite{KaWe}. The only difference between $\operatorname{Rep} U_\zeta(\mathfrak{sl}_2)$ and its non-trivial $3$-cocycle twist is that the $3$-cocycle twist has a new associativity isomorphism $\widetilde{\mathcal{A}}$ characterized by
\begin{gather}\label{eqn:3-cocycle-assoc}
\widetilde{\mathcal{A}}_{V_{n_1},V_{n_2},V_{n_3}} =(-1)^{n_1n_2n_3}\mathcal{A}_{V_{n_1},V_{n_2},V_{n_3}},
\end{gather}
where $V_n$ is the $(n+1)$-dimensional simple object of $\operatorname{Rep} U_\zeta(\mathfrak{sl}_2)$. It is not difficult to use~\cite[Theorem $A_{\infty}$]{KaWe} to see that the non-trivial $3$-cocycle twist of $\operatorname{Rep} U_\zeta(\mathfrak{sl}_2)$ is tensor equivalent to~$\operatorname{Rep} U_{-\zeta}(\mathfrak{sl}_2)$, so $\mathcal{C}_{\mathfrak{sl}_2}$ is tensor equivalent to $\operatorname{Rep} U_\zeta(\mathfrak{sl}_2)$ for either $\zeta=\pm 1$ or $\zeta$ not a root of unity. We need to rule out the latter possibility.

From \cite[Proposition 6.3]{GaNe}, the tensor category $\operatorname{Rep} U_\zeta(\mathfrak{sl}_2)$ admits two or four braidings characterized by the value of $\mathcal{R}_{V_1,V_1}$.
Thus $\mathcal{C}_{\mathfrak{sl}_2}$ also admits two or four braidings, and $\mathcal{R}_{\mathcal{L}_1,\mathcal{L}_1}$ in particular is given by
\begin{gather*}
\mathcal{R}_{\mathcal{L}_1,\mathcal{L}_1}^{\pm 1} = \pm \bigl(-\zeta^{3/2}\cdot i_{11}^0\circ\pi_{11}^0 +\zeta^{-1/2}\cdot i_{11}^2\circ\pi_{11}^2\bigr)
\end{gather*}
for one of the four possible choices of signs, using the notation of \eqref{eqn:proj-system}. In particular,
\begin{gather}\label{eqn:qg-braiding}
\mathcal{R}_{\mathcal{L}_1,\mathcal{L}_1}^2 =\zeta^{\pm 3}\cdot i_{11}^0\circ\pi_{11}^0 +\zeta^{\mp 1}\cdot i_{11}^2\circ\pi_{11}^2.
\end{gather}
On the other hand, the balancing equation
\begin{gather*}
{\rm e}^{2\pi{\rm i }L_0}=\mathcal{R}^2_{\mathcal{L}_1,\mathcal{L}_1}\circ\bigl({\rm e}^{2\pi{\rm i }L_0}\boxtimes {\rm e}^{2\pi{\rm i }L_0}\bigr)
\end{gather*}
for vertex algebraic tensor categories (see, for example, \cite[Theorem 4.1]{Hu-rigidity}) implies that
\begin{gather}\label{eqn:va-braiding}
\mathcal{R}_{\mathcal{L}_1,\mathcal{L}_1}^2 = e^{-4\pi {\rm i} h_{3p-1,1}}\bigl(i^0_{11}\circ\pi^0_{11}+e^{2\pi {\rm i}h_{4p-1,1}}\cdot i^2_{11}\circ\pi^2_{11}\bigr) = (-1)^{pq}\mathrm{Id}_{\mathcal{L}_1\boxtimes\mathcal{L}_1}.
\end{gather}
Comparing \eqref{eqn:qg-braiding} and \eqref{eqn:va-braiding}, we get $\zeta =(-1)^{pq}$. Thus $\mathcal{C}_{\mathfrak{sl}_2}$ is tensor equivalent to $\operatorname{Rep} U_{\pm 1}(\mathfrak{sl}_2)$, and the tensor equivalence is also braided if we equip $\operatorname{Rep} U_{\pm 1}(\mathfrak{sl}_2)$ with the appropriate one of its two braidings.

Since $\operatorname{Rep} U_{-1}(\mathfrak{sl}_2)$ is tensor equivalent to the non-trivial $3$-cocycle twist of $\operatorname{Rep} U_1(\mathfrak{sl}_2)=\operatorname{Rep} \mathfrak{sl}_2$, and since \eqref{eqn:3-cocycle-assoc} implies that the $3$-cocycle twist does not affect the associativity isomorphisms of the tensor subcategory $\operatorname{Rep} \mathrm{PSL}_2\subseteq \operatorname{Rep}\mathfrak{sl}_2$, it follows that the subcategory $\mathcal{C}_{\mathrm{PSL}_2}\subseteq\mathcal{C}_{\mathfrak{sl}_2}$ is tensor equivalent to $\operatorname{Rep} \mathrm{PSL}_2$.
Moreover, this is an equivalence of braided tensor categories if we equip $\operatorname{Rep} \mathrm{PSL}_2$ with the restriction of some suitable braiding on $\operatorname{Rep} \mathfrak{sl}_2$ or its $3$-cocycle twist.
Let $\mathcal{R}$ be the standard braiding on $\operatorname{Rep} \mathfrak{sl}_2$, which restricts to the standard braiding on~$\operatorname{Rep} \mathrm{PSL}_2$. Then the second braiding $\widetilde{\mathcal{R}}$ on~$\operatorname{Rep} \mathfrak{sl}_2$ is given by
\begin{gather*}
\widetilde{\mathcal{R}}_{V_{n_1},V_{n_2}} = (-1)^{n_1 n_2}\mathcal{R}_{V_{n_1},V_{n_2}},
\end{gather*}
while the two braidings $\widetilde{\mathcal{R}}$ on the non-trivial $3$-cocycle twist of $\operatorname{Rep} \mathfrak{sl}_2$ are given by
\begin{gather*}
\widetilde{\mathcal{R}}_{V_{n_1},V_{n_2}} =\begin{cases}
\pm {\rm i}\cdot\mathcal{R}_{V_{n_1},V_{n_2}} & \text{if } n_1 n_2 \text{ is odd},\\
\mathcal{R}_{V_{n_1},V_{n_2}} & \text{if } n_1 n_2 \text{ is even},\\
\end{cases}
\end{gather*}
where ${\rm i}$ is a square root of $-1$; to see why, simply note that all these braidings satisfy the hexagon axiom (keeping in mind \eqref{eqn:3-cocycle-assoc} in the $3$-cocycle twist case), and thus they must comprise all the braidings from \cite[Proposition 6.3]{GaNe}. These braidings on $\operatorname{Rep} \mathfrak{sl}_2$ and its $3$-cocycle twist all restrict to the standard braiding on $\operatorname{Rep} \mathrm{PSL}_2$, so the tensor equivalence between $\mathcal{C}_{\mathfrak{sl}_2}$ and~$\operatorname{Rep} \mathrm{PSL}_2$ preserves braidings.
\end{proof}

\begin{Remark}
In \cite{Del}, Deligne showed that any rigid symmetric tensor category of moderate growth over an algebraically closed field $\mathbb{K}$ of characteristic $0$ is super-Tannakian, that is, equivalent to $\operatorname{Rep}(G,z)$ for some affine supergroup scheme $G$ and suitable element $z\in G(\mathbb{K})$ of order~$2$. For $\mathcal{C}_{\mathrm{PSL}_2}$, it is easy to calculate directly
 from the balancing equation that the braiding is symmetric, and the fusion rules \eqref{eqn:unif-fus-rules} easily imply that $\mathcal{C}_{\mathrm{PSL}_2}$ has moderate growth (see, for example, the exposition \cite[Section 2.6]{EtKa} for the definition of moderate growth). Thus another way to prove Theorem~\ref{thm:braid-equiv} would be to show that $\operatorname{Rep} \mathrm{PSL}_2$ is characterized as a super-Tannakian category by its fusion rules.
\end{Remark}

\section{Commutative algebras and vertex operator algebras}\label{sec:comm-alg-and-VOA}

Let $(\mathcal{C},\boxtimes,\mathbf{1},\mathcal{A},l,r,\mathcal{R})$ be a braided tensor category. A \textit{commutative algebra} $(A,\mu_A,\iota_A)$ in $\mathcal{C}$ is an object $A$ equipped with morphisms
$
\mu_A\colon A\boxtimes A\longrightarrow A$, $\iota_A\colon \mathbf{1}\longrightarrow A$
satisfying the following properties:
\begin{enumerate}\itemsep=0pt
\item[(1)] \textit{Unitality}: $\mu_A\circ(\iota_A\boxtimes\mathrm{Id}_A)=l_A$ and $\mu_A\circ(\mathrm{Id}_A\boxtimes\iota_A)=r_A$.

\item[(2)] \textit{Associativity}: $\mu_A\circ(\mathrm{Id}_A\boxtimes\mu_A)=\mu_A\circ(\mu_A\boxtimes\mathrm{Id}_A)\circ\mathcal{A}_{A,A,A}$.

\item[(3)] \textit{Commutativity}: $\mu_A =\mu_A\circ\mathcal{R}_{A,A}$.
\end{enumerate}
Since $A$ is commutative, ideals in $A$ are the same as left ideals, and a left ideal is a subobject~${I\subseteq A}$ such that $\operatorname{Im} \mu_A\vert_{A\boxtimes I}\subseteq I$. A commutative algebra $A$ in $\mathcal{C}$ is \textit{simple} if its only (left) ideals are $0$ and $A$.

If $(A,\mu_A,\iota_A)$ and $(B,\mu_B,\iota_B)$ are two commutative algebras in $\mathcal{C}$, then an algebra isomorphism~${(A,\mu_A,\iota_A)\rightarrow(B,\mu_B,\iota_B)}$ is a $\mathcal{C}$-isomorphism $g\colon A\rightarrow B$ such that
$
g\circ\iota_A=\iota_B$, $ g\circ\mu_A=\mu_B\circ(g\boxtimes g)$.
Let $\mathrm{Aut}_\mathcal{C}(A)$ be the group of automorphisms of the commutative algebra $(A,\mu_A,\iota_A)$ in $\mathcal{C}$.

If $\mathcal{C}$ is a braided tensor category of modules for a VOA $V$, then a commutative algebra~$A$ in~$\mathcal{C}$ with an injective unit map $\iota_A$ is the same thing as a VOA $A$ which contains $V$ as a~vertex operator subalgebra and which is an object of $\mathcal{C}$ when considered as a $V$-module \cite{HKL}. In this setting, the relation between the multiplication map $\mu_A\colon A\boxtimes A\rightarrow A$ and the vertex operator~${Y_A\colon A\otimes A\rightarrow A((x))}$ is
$
\mu_A\circ\mathcal{Y}_\boxtimes = Y_A$,
where $\mathcal{Y}_\boxtimes$ is the tensor product intertwining operator of type $\binom{A}{A A}$.

In this paper, we are particularly concerned with the VOA extension $V_{c_{p,q}}\hookrightarrow W_{p,q}$ for coprime~${p,q\in\mathbb{Z}_{\geq 2}}$, where $W_{p,q}$ is the triplet VOA introduced in \cite{FGST1-Log}. The structure of $W_{p,q}$ as a~$V_{c_{p,q}}$-module follows from \cite[Definition 4.1 and Lemma 3.5.2]{FGST1-Log}; see also \cite[Proposition 5.4]{AM-W2p-alg}, \cite[Section 4]{AM-C2-W-alg}, \cite[Proposition 4.14 and Definition 5.1]{TW2}. Namely,
\begin{gather}\label{eqn:Wpq}
W_{p,q}\cong\mathcal{K}_{1,1}\oplus\bigoplus_{n=2}^\infty (2n-1)\cdot \mathcal{L}_{2np-1,1}
\end{gather}
as a $\mathcal{V}{\rm ir}$-module. Moreover, $W_{p,q}$ is not simple, since by \cite[Theorem 5.4]{TW2} there is non-split short exact sequence of $W_{p,q}$-modules
\begin{gather}\label{eqn:Wpq-exact-seq}
0\longrightarrow I_{p,q}\longrightarrow W_{p,q}\longrightarrow L_{c_{p,q}}\longrightarrow 0.
\end{gather}
Here $I_{p,q}$ is a simple ideal such that $I_{p,q}\cong\bigoplus_{n=1}^\infty (2n-1)\cdot \mathcal{L}_{2np-1,1}$ as a $\mathcal{V}{\rm ir}$-module, and $L_{c_{p,q}}$ is the simple Virasoro VOA of central charge $c_{p,q}$, which is isomorphic to $\mathcal{L}_{1,1}$ as a $\mathcal{V}{\rm ir}$-module. Note that $W_{p,q}$ is not an object of $\mathcal{O}_{c_{p,q}}$ since it has infinite length as a $\mathcal{V}{\rm ir}$-module, but it is an object of the direct limit completion, or ind-category, $\operatorname{Ind}(\mathcal{O}_{c_{p,q}})$, which is also a vertex algebraic braided tensor category by the main theorem of \cite{CMY1}. Thus $W_{p,q}$ is a commutative algebra in~$\operatorname{Ind}(\mathcal{O}_{c_{p,q}})$ by the main results of \cite{HKL}.

In the rest of this section, we will give a new construction of $W_{p,q}$. Without assuming that~\eqref{eqn:Wpq} already admits a VOA structure, we will use tensor category methods to construct a~VOA structure on the Virasoro direct sum on the right-hand side of \eqref{eqn:Wpq} such that $I_{p,q}$ is still a simple ideal. We will then show that any such VOA structure on \eqref{eqn:Wpq} is unique up to isomorphism, and therefore our construction yields the same VOA as the triplet algebra~$W_{p,q}$ constructed in~\cite{FGST1-Log}. Our construction will make it obvious that the group $\mathrm{Aut}(W_{p,q})$ of VOA automorphisms of $W_{p,q}$ is $\mathrm{PSL}_2(\mathbb{C})$. This result on $\mathrm{Aut}(W_{p,q})$ was proved for the $q=1$ case in~\cite{ALM}, and some of the work in~\cite{FGST1-Log} and \cite{TW2} suggested that the same result should hold in general, although these papers did not give full proofs in the~$q\geq 2$ case. To construct $W_{p,q}$, we will first use Theorem~\ref{thm:braid-equiv} to obtain a simple commutative algebra $W_{p,q}'$ in $\mathcal{C}_{\mathrm{PSL}_2}$ on which $\mathrm{PSL}_2(\mathbb{C})$ acts by automorphisms, such that
\begin{gather}\label{eqn:Wpq'}
W_{p,q}' \cong\mathcal{K}_{1,1}'\oplus\bigoplus_{n=2}^\infty V_{2n-2}\otimes\mathcal{L}_{2np-1,1} \cong\bigoplus_{n=0}^\infty V_{2n}\otimes\mathcal{L}_{2n}
\end{gather}
as a $\mathrm{PSL}_2(\mathbb{C})\times \mathcal{V}{\rm ir}$-module; here as before, $V_n$ denotes the $(n+1)$-dimensional simple $\mathfrak{sl}_2$-module, which is a $\mathrm{PSL}_2(\mathbb{C})$-module if and only if $n$ is even. Then we will use the non-zero $\mathcal{V}{\rm ir}$-homomorphism $\varphi\colon \mathcal{K}_{1,1}'\rightarrow\mathcal{K}_{1,1}$ to transfer this simple algebra structure on $W_{p,q}'$ to a non-simple algebra structure on the $\mathcal{V}{\rm ir}$-module $W_{p,q}$ of \eqref{eqn:Wpq} such that the $\mathcal{V}{\rm ir}$-submodule $I_{p,q}$ from~\eqref{eqn:Wpq-exact-seq} is a simple ideal.

\begin{Proposition}\label{prop2}
There is a unique $($up to isomorphism$)$ simple commutative algebra structure on the object
\begin{gather}\label{eqn:Wpq'-forget-decomp}
W_{p,q}' =\bigoplus_{n=0}^\infty (2n+1)\cdot\mathcal{L}_{2n}
\end{gather}
of $\operatorname{Ind}(\mathcal{C}_{\mathrm{PSL}_2})$. Moreover, $\mathrm{Aut}_{\operatorname{Ind}(\mathcal{C}_{\mathrm{PSL}_2})}(W_{p,q}')\cong \mathrm{PSL}_2(\mathbb{C})$ and $W_{p,q}'$ has the decomposition \eqref{eqn:Wpq'} as a $\mathrm{PSL}_2(\mathbb{C})\times\mathcal{V}{\rm ir}$-module.
\end{Proposition}
\begin{proof}
Since $\mathcal{C}_{\mathrm{PSL}_2}$ is a rigid symmetric tensor category equivalent to $\operatorname{Rep} \mathrm{PSL}_2$ by Theorem~\ref{thm:braid-equiv}, we can ``glue'' $\operatorname{Rep} \mathrm{PSL}_2$ and $\mathcal{C}_{\mathrm{PSL}_2}$ as in \cite[Main Theorem 1]{CKM2} to obtain a simple commutative algebra in the ind-category of the Deligne tensor product category $\operatorname{Rep} \mathrm{PSL}_2\otimes\mathcal{C}_{\mathrm{PSL}_2}$ with the decomposition \eqref{eqn:Wpq'}. This algebra is essentially the canonical algebra of
$\operatorname{Rep} \mathrm{PSL}_2$; see, for example, \cite[Section 7.9]{EGNO}, and compare also with the Peter--Weyl Theorem for the compact real form~$SO_3(\mathbb{R})$ of $\mathrm{PSL}_2(\mathbb{C})$.
Applying the (forgetful) fiber functor from $\operatorname{Rep} \mathrm{PSL}_2$ to the category~$\mathcal{V}ec$ of finite-dimensional vector spaces and observing that~$\mathcal{V}ec\otimes\mathcal{C}_{\mathrm{PSL}_2}\cong\mathcal{C}_{\mathrm{PSL}_2}$ as symmetric tensor categories, we get a commutative algebra $W_{p,q}'$ in $\operatorname{Ind}(\mathcal{C}_{\mathrm{PSL}_2})$ with the decomposition~\eqref{eqn:Wpq'-forget-decomp}; see, for example, \cite[Appendix A]{MY-25-Vir}.

Since $W_{p,q}'$ is obtained from a commutative algebra in $\operatorname{Ind}(\operatorname{Rep} \mathrm{PSL}_2\otimes\mathcal{C}_{\mathrm{PSL}_2})$, the algebra multiplication $\mu_{W_{p,q}'}\colon W_{p,q}'\boxtimes W_{p,q}'\rightarrow W_{p,q}'$ is a $\mathrm{PSL}_2(\mathbb{C})$-module homomorphism if we give the~$(2n+1)$-dimensional multiplicity space of each $\mathcal{L}_{2n}$ in $W_{p,q}'$ the structure of the $\mathrm{PSL}_2(\mathbb{C})$-module $V_{2n}$ as in \eqref{eqn:Wpq'}. This is equivalent to saying that $\mathrm{PSL}_2(\mathbb{C})$ acts on $W_{p,q}'$ by algebra automorphisms. Moreover, since the original commutative algebra in $\operatorname{Ind}(\operatorname{Rep} \mathrm{PSL}_2\otimes\mathcal{C}_{\mathrm{PSL}_2})$ is simple, $W_{p,q}'$ has no non-zero proper $\mathrm{PSL}_2(\mathbb{C})$-invariant ideals. We claim that this implies $W_{p,q}'$ has no non-zero proper ideals and thus is simple as an algebra in $\operatorname{Ind}(\mathcal{C}_{\mathrm{PSL}_2})$.

The proof of the claim is similar to part of the proof of \cite[Proposition C.1]{MY-25-Vir}. First, since~$\operatorname{Ind}(\mathcal{C}_{\mathrm{PSL}_2})$ is semisimple, any non-zero ideal of $W_{p,q}'$ contains a copy of $\mathcal{L}_{2n}$ for some ${n\in\mathbb{Z}_{\geq 0}}$. Thus it is enough to show that for any $n\in\mathbb{Z}_{\geq 0}$ and non-zero $v\in V_{2n}$, the ideal generated by~${v\otimes\mathcal{L}_{2n}\subseteq W_{p,q}'}$ contains $\mathcal{L}_0$, since $\mu_{W_{p,q}'}\vert_{W_{p,q}'\boxtimes\mathcal{L}_0} = r_{W_{p,q}'}'$ is surjective.
To prove this, let
\begin{gather*}
i_m\colon\ V_{2m}\otimes\mathcal{L}_{2m}\longrightarrow W_{p,q}',\qquad \pi_m\colon\ W_{p,q}'\longrightarrow V_{2m}\otimes\mathcal{L}_{2m}
\end{gather*}
for $m\in\mathbb{Z}_{\geq 0}$ be the obvious inclusion and projection morphisms in $\operatorname{Ind}(\mathcal{C}_{\mathrm{PSL}_2})$. Since $V_{2n}\otimes\mathcal{L}_{2n}$ is $\mathrm{PSL}_2(\mathbb{C})$-invariant, it generates a $\mathrm{PSL}_2(\mathbb{C})$-invariant ideal which must be all of $W_{p,q}'$. This implies in particular that for any $n\in\mathbb{Z}_{\geq 0}$,
$
\pi_0\circ\mu_{W_{p,q}'}\circ(i_m\boxtimes i_n)\neq 0
$
for some $m$, and in fact $m=n$ since $\mathcal{L}_m\boxtimes\mathcal{L}_n$ contains $\mathcal{L}_0$ only if $m=n$. Then since $\mu_{W_{p,q}'}$ is a~$\mathrm{PSL}_2(\mathbb{C})$-homomorphism,
\begin{gather}\label{eqn:contain-L0}
\pi_0\circ\mu_{W_{p,q}'}\circ(i_m\boxtimes i_n) =(\cdot,\cdot)_{2n}\otimes\pi_{2n,2n}^0,
\end{gather}
where $(\cdot,\cdot)_{2n}$ is the unique (up to scaling) non-degenerate $\mathrm{PSL}_2(\mathbb{C})$-invariant bilinear form on~$V_{2n}$ and $\pi_{2n,2n}^0\colon \mathcal{L}_{2n}\boxtimes\mathcal{L}_{2n}\rightarrow\mathcal{L}_0$ is as in \eqref{eqn:proj-system} (and is surjective). Now we want the ideal generated by~${v\otimes\mathcal{L}_{2n}}$ to contain $\mathcal{L}_0$ for any non-zero $v\in V_{2n}$. In fact, taking $v'\in V_{2n}$ such that $(v',v)_{2n}\neq 0$, \eqref{eqn:contain-L0} implies that
$
\operatorname{Im} \pi_0\circ\mu_{W_{p,q}'}\vert_{(v'\otimes\mathcal{L}_{2n})\boxtimes (v\otimes\mathcal{L}_{2n})} =\mathcal{L}_0$.
Since the ideal generated by $v\otimes\mathcal{L}_{2n}$ is semisimple, it thus contains $\mathcal{L}_0$ as required. This completes the proof that $W_{p,q}'$ is a simple algebra in $\operatorname{Ind}(\mathcal{C}_{\mathrm{PSL}_2})$.

Finally, to show that the simple commutative algebra structure on \eqref{eqn:Wpq'-forget-decomp} is unique up to isomorphism and that $\mathrm{PSL}_2(\mathbb{C})$ is the full automorphism group of $W_{p,q}'$, we may replace $\mathcal{C}_{\mathrm{PSL}_2}$ with any symmetric tensor category $\mathcal{C}$ equivalent to $\operatorname{Rep} \mathrm{PSL}_2$ and then replace $W_{p,q}'$ with the corresponding simple algebra in $\operatorname{Ind}(\mathcal{C})$. That is, it is enough to find $\mathcal{C}$ such that any simple commutative algebra in $\operatorname{Ind}(\mathcal{C})$ of the form
\begin{gather*}
A=\bigoplus_{n=0}^\infty (2n+1)\cdot X_{2n},
\end{gather*}
where $X_{2n}$ is the image of $V_{2n}$ under a symmetric tensor equivalence $\operatorname{Rep} \mathrm{PSL}_2\rightarrow\mathcal{C}$, is unique up to isomorphism and has automorphism group $\mathrm{PSL}_2(\mathbb{C})$. See, for example, \cite[Appendix A]{MY-25-Vir} and the last paragraph in the proof of \cite[Theorem 7.1]{MY-25-Vir} for why this is sufficient.

One possibility for $\mathcal{C}$ is a subcategory of the braided tensor category $\mathcal{O}_1$ of $C_1$-cofinite modules for the Virasoro VOA of central charge $1$. Namely, we take $\mathcal{C}$ to be the full subcategory of~$\mathcal{O}_1$ whose objects are isomorphic to finite direct sums of simple $\mathcal{V}{\rm ir}$-modules $X_{2n}$ of central charge~$1$ and lowest conformal weight \smash{$h^{(1)}_{2n+1,1} = n^2$}. It is shown that $\mathcal{C}\cong\operatorname{Rep} \mathrm{PSL}_2$ as symmetric tensor categories in \cite[Example 4.12]{McR-cpt-orb}. Moreover, \cite[Theorem B.1]{MY-25-Vir} shows that any simple VOA, equivalently simple commutative algebra, $A$ in $\operatorname{Ind}(\mathcal{C})$ such that $A\cong\bigoplus_{n=0}^\infty (2n+1)\cdot X_{2n}$ as a~$\mathcal{V}{\rm ir}$-module is isomorphic to the $\mathfrak{sl}_2$-root lattice VOA $V_{\sqrt{2}\mathbb{Z}}$. Now, the group of algebra automorphisms of $V_{\sqrt{2}\mathbb{Z}}$ is the group of VOA automorphisms that fix $X_0$; but this is the group of all VOA automorphisms since $X_0$ is the Virasoro vertex operator subalgebra of $V_{\sqrt{2}\mathbb{Z}}$ and hence is fixed by any automorphism. Then since $V_{\sqrt{2}\mathbb{Z}}$ is isomorphic to the simple affine VOA of~$\mathfrak{sl}_2$ at level $1$, which is generated by its conformal weight $1$ space that is a Lie algebra isomorphic to $\mathfrak{sl}_2$, the automorphism group of $V_{\sqrt{2}\mathbb{Z}}$ is isomorphic to $\mathrm{Aut}(\mathfrak{sl}_2) =\mathrm{Ad}(SL_2(\mathbb{C})) =\mathrm{PSL}_2(\mathbb{C})$. This completes the proof of the proposition.
\end{proof}

Now we need to adjust the simple commutative algebra structure on $W_{p,q}'$ in $\operatorname{Ind}(\mathcal{C}_{\mathrm{PSL}_2})$ to get a non-simple commutative algebra structure on the $\mathcal{V}{\rm ir}$-module $W_{p,q}$ in \eqref{eqn:Wpq}, which is an object of $\operatorname{Ind}(\mathcal{O}_{c_{p,q}})$. For future applications, we work in a general situation. Let $(\mathcal{C},\boxtimes,\mathbf{1},\mathcal{A},l,r,\mathcal{R})$ be a braided tensor category, and let $\mathcal{C}'\subseteq\mathcal{C}$ be a full subcategory which is closed under $\boxtimes$. We assume $\mathcal{C}'$ has an object $\mathbf{1}'$ together with a morphism $\varphi\colon \mathbf{1}'\rightarrow\mathbf{1}$ in $\mathcal{C}$ such that
\begin{gather}\label{eqn:l'-r'-gen-def}
l_X' = l_X\circ(\varphi\boxtimes\mathrm{Id}_X),\qquad r_X'=r_X\circ(\mathrm{Id}_X\boxtimes\varphi)
\end{gather}
are isomorphisms for any object $X$ in $\mathcal{C}'$ (like in \eqref{eqn:ln'-rn'-def}), and such that
 $(\mathcal{C}',\boxtimes,\mathcal{A},l',r',\mathcal{R})$ is a~braided tensor category. For example, we could take $\mathcal{C}=\operatorname{Ind}(\mathcal{O}_{c_{p,q}})$ and $\mathcal{C}'=\operatorname{Ind}(\mathcal{C}_{\mathrm{PSL}_2})$.

Now suppose we have objects $A$ and $A'$ in $\mathcal{C}$ and $\mathcal{C}'$, respectively, equipped with a morphism~${\Phi\colon A\rightarrow A'}$. For example, if $A=\mathbf{1}\oplus J$ and $A'=\mathbf{1}'\oplus J$ for some object $J$ in $\mathcal{C}'$, then we could take $\Phi=\varphi\oplus\mathrm{Id}_J$. We also assume that the maps
\begin{gather}\label{eqn:A0-iso}
\Phi\circ -\colon\ \operatorname{Hom}(\mathbf{1}',A')\longrightarrow\operatorname{Hom}(\mathbf{1}',A),\qquad -\circ\varphi\colon\ \operatorname{Hom}(\mathbf{1},A)\longrightarrow\operatorname{Hom}(\mathbf{1}',A)
\end{gather}
and
\begin{gather}
\Phi\circ -\colon\ \operatorname{Hom}\bigl(\bigl(A'\bigr)^{\boxtimes n},A'\bigr)\longrightarrow\operatorname{Hom}\bigl(\bigl(A'\bigr)^{\boxtimes n},A\bigr),\nonumber\\ -\circ\Phi^{\boxtimes n}\colon\ \operatorname{Hom}\bigl(A^{\boxtimes n},A\bigr)\longrightarrow\operatorname{Hom}\bigl(\bigl(A'\bigr)^{\boxtimes n},A\bigr)\label{eqn:An-iso}
\end{gather}
for $n=1,2,3$ are isomorphisms. Note that these maps would obviously be isomorphisms if $\Phi$ and $\varphi$ were isomorphisms in $\mathcal{C}$, but we are \textit{not} assuming this. Note also that \eqref{eqn:A0-iso} and the $n=2$ case of \eqref{eqn:An-iso} yield isomorphisms
\begin{align*}
\begin{aligned}
\operatorname{Hom}(\mathbf{1}',A') & \longrightarrow\operatorname{Hom}(\mathbf{1},A),\\
\iota_{A'} & \longmapsto \iota_A,
\end{aligned}\qquad
\begin{aligned}
 \operatorname{Hom}(A'\boxtimes A',A') & \longrightarrow\operatorname{Hom}(A\boxtimes A,A),\\
\mu_{A'} & \longmapsto \mu_A,
\end{aligned}
\end{align*}
such that
\begin{gather}\label{eqn:A'-A-structure-reln}
\Phi\circ\iota_{A'} =\iota_A\circ\varphi,\qquad \Phi\circ\mu_{A'} =\mu_A\circ(\Phi\boxtimes\Phi).
\end{gather}
Similarly, the $n=1$ case of \eqref{eqn:An-iso} yields an isomorphism
\begin{gather*}
\operatorname{Hom}(A',A') \longrightarrow\operatorname{Hom}(A,A),\qquad
g' \longmapsto g
\end{gather*}
such that
\begin{gather}\label{eqn:g'-g-reln}
\Phi\circ g' = g\circ\Phi.
\end{gather}
With this setup, we now prove the following.

\begin{Theorem}\label{thm:C'-alg-to-C-alg}
In the setting of the previous paragraph,
\begin{enumerate}\itemsep=0pt
\item[$(1)$] $(A',\mu_{A'},\iota_{A'})$ is a commutative algebra in $\mathcal{C}'$ if and only if $(A,\mu_A,\iota_A)$ is a commutative algebra in $\mathcal{C}$, where $(\iota_{A'},\mu_{A'})$ and $(\iota_A,\mu_A)$ are related by \eqref{eqn:A'-A-structure-reln}.

\item[$(2)$] A $\mathcal{C}'$-morphism $g'\colon A'\rightarrow A'$ is an algebra isomorphism between two commutative algebra structures \smash{$\bigl(A',\mu_{A'}^{(1)},\iota_{A'}^{(1)}\bigr)$} and \smash{$\bigl(A',\mu_{A'}^{(2)},\iota_{A'}^{(2)}\bigr)$} if and only if $g\colon A\rightarrow A$ defined by \eqref{eqn:g'-g-reln} is an~algebra isomorphism between \smash{$\bigl(A,\mu_{A}^{(1)},\iota_{A}^{(1)}\bigr)$} and \smash{$\bigl(A,\mu_{A}^{(2)},\iota_{A}^{(2)}\bigr)$}, where \smash{$\bigl(\mu_A^{(i)},\iota_A^{(i)}\bigr)$} for $i=1,2$ are defined by \eqref{eqn:A'-A-structure-reln}.
\end{enumerate}
\end{Theorem}

\begin{proof}
(1) For the left unit property of a commutative algebra, we calculate
\begin{align*}
\mu_A\circ(\iota_A\boxtimes\mathrm{Id}_A)\circ l_A^{-1}\circ\Phi & = \mu_A\circ(\iota_A\boxtimes\mathrm{Id}_A)\circ(\mathrm{Id}_{\mathbf{1}}\boxtimes\Phi)\circ l_{A'}^{-1}\nonumber\\
& =\mu_A\circ(\mathrm{Id}_{A}\boxtimes\Phi)\circ(\iota_A\boxtimes\mathrm{Id}_{A'})\circ(\varphi\boxtimes\mathrm{Id}_{A'})\circ(l_{A'}')^{-1}\nonumber\\
& =\mu_A\circ(\Phi\boxtimes\Phi)\circ(\iota_{A'}\boxtimes\mathrm{Id}_{A'})\circ(l_{A'}')^{-1} \nonumber\\
&=\Phi\circ\mu_{A'}\circ(\iota_{A'}\boxtimes\mathrm{Id}_{A'})\circ(l_{A'}')^{-1}
\end{align*}
using the naturality of the left unit isomorphisms in $\mathcal{C}$, \eqref{eqn:l'-r'-gen-def}, and \eqref{eqn:A'-A-structure-reln}.
Since the two maps in~\eqref{eqn:An-iso} are isomorphisms in the $n=1$ case, it follows that
\begin{gather*}
\mu_A\circ(\iota_A\boxtimes\mathrm{Id}_A)\circ l_A^{-1} =\mathrm{Id}_A\longleftrightarrow \mu_{A'}\circ(\iota_{A'}\boxtimes\mathrm{Id}_{A'})\circ(l_{A'}')^{-1}=\mathrm{Id}_{A'}.
\end{gather*}
Thus the left unit property holds for $(A',\mu_{A'},\iota_{A'})$ if and only if it holds for $(A,\mu_A,\iota_A)$.

It follows similarly that the right unit property for $(A',\mu_{A'},\iota_{A'})$ is equivalent to the right unit property for $(A,\mu_A,\iota_A)$. Alternatively, this follows from the left unit property and the equivalence of the commutativity of $\mu_{A'}$ and $\mu_A$, which we prove next. Since
\begin{align*}
\mu_A\circ\mathcal{R}_{A,A}\circ(\Phi\boxtimes\Phi) =\Phi\circ\mu_{A'}\circ\mathcal{R}_{A',A'},\qquad\mu_A\circ(\Phi\boxtimes\Phi)=\Phi\circ\mu_{A'}
\end{align*}
by \eqref{eqn:A'-A-structure-reln} and the naturality of the braiding isomorphisms, the assumption that the two maps in~\eqref{eqn:An-iso} are isomorphisms in the $n=2$ case implies that $\mu_A\circ\mathcal{R}_{A,A}=\mu_A$ if and only if $\mu_{A'}\circ\mathcal{R}_{A',A'}=\mu_{A'}$. Thus $\mu_A$ is commutative if and only if $\mu_{A'}$ is. Similarly, $\mu_A$ is associative if and only if $\mu_{A'}$ is because
\begin{align*}
\mu_A\circ(\mathrm{Id}_A\boxtimes\mu_A)\circ(\Phi\boxtimes(\Phi\boxtimes\Phi)) =\Phi\circ\mu_{A'}\circ(\mathrm{Id}_{A'}\boxtimes\mu_{A'})
\end{align*}
and
\begin{gather*}
\mu_A\circ(\mu_A\boxtimes\mathrm{Id}_A)\circ\mathcal{A}_{A,A,A}\circ(\Phi\boxtimes(\Phi\boxtimes\Phi)) = \Phi\circ\mu_{A'}\circ(\mu_{A'}\boxtimes\mathrm{Id}_{A'})\circ\mathcal{A}_{A',A',A'}
\end{gather*}
by \eqref{eqn:A'-A-structure-reln} and naturality of the associativity isomorphisms, and because the two maps in \eqref{eqn:An-iso} are isomorphisms in the $n=3$ case.
This proves the first statement of the theorem.

(2) Now suppose $g'\colon A'\rightarrow A'$ and $g\colon A\rightarrow A$ are two morphisms related by~\eqref{eqn:g'-g-reln}. If $g$ is a~$\mathcal{C}$-isomorphism with inverse $g^{-1}$, then $g'$ is a $\mathcal{C}'$-isomorphism with inverse $\bigl(g^{-1}\bigr)'$ characterized by \smash{$\Phi\circ\bigl(g^{-1}\bigr)'=g^{-1}\circ\Phi$} as in \eqref{eqn:g'-g-reln}. Indeed,
\begin{gather*}
\Phi\circ g'\circ\bigl(g^{-1}\bigr)' =g\circ g^{-1}\circ\Phi =\Phi
\end{gather*}
by \eqref{eqn:g'-g-reln}, and therefore \smash{$g'\circ\bigl(g^{-1}\bigr)'=\mathrm{Id}_{A'}$} since the first map of \eqref{eqn:An-iso} is an isomorphism in the $n=1$ case. Similarly, \smash{$\bigl(g^{-1}\bigr)'\circ g' =\mathrm{Id}_{A'}$}, and similarly $g$ is a~$\mathcal{C}$-isomorphism if $g'$ is a~$\mathcal{C}'$-isomorphism.

Now we consider how $g$ and $g'$ relate to different algebra structures on $A$ and $A'$. We have
\begin{gather*}
g\circ\iota^{(1)}_A\circ\varphi =\Phi\circ g'\circ\iota_{A'}^{(1)},\qquad \iota_{A}^{(2)}\circ\varphi =\Phi\circ\iota_{A'}^{(2)}
\end{gather*}
and
\begin{gather*}
g\circ\mu_A^{(1)}\circ(\Phi\boxtimes\Phi) =\Phi\circ g'\circ\mu_{A'}^{(1)},\qquad \mu_{A}^{(2)}\circ(g\boxtimes g)\circ(\Phi\boxtimes\Phi) =\Phi\circ\mu_{A'}^{(2)}\circ(g'\boxtimes g')
\end{gather*}
by \eqref{eqn:A'-A-structure-reln} and \eqref{eqn:g'-g-reln}. Thus
\begin{gather*}
g\circ\iota_A^{(1)}=\iota_A^{(2)}\longleftrightarrow g'\circ\iota_{A'}^{(1)} =\iota_{A'}^{(2)}
\end{gather*}
since the maps in \eqref{eqn:A0-iso} are an isomorphism, and
\begin{gather*}
g\circ\mu_A^{(1)}=\mu_A^{(2)}\circ(g\boxtimes g) \longleftrightarrow g'\circ\mu_{A'}^{(1)}=\mu_{A'}^{(2)}\circ(g'\boxtimes g')
\end{gather*}
since the maps in \eqref{eqn:An-iso} are isomorphisms in the $n=2$ case.
This proves that $g\colon \smash{\bigl(A,\mu\raisebox{1pt}{${}_A^{(1)}$},\iota\raisebox{1pt}{${}_A^{(1)}$}\bigr)}\!\rightarrow \smash{\bigl(A,\mu\raisebox{-1pt}{${}_A^{(2)}$}, \iota\raisebox{-1pt}{${}_A^{(2)}$}\bigr)}$ is an isomorphism of $\mathcal{C}$-algebras if and only if \smash{$g'\colon \bigl(A',\mu\raisebox{-1pt}{${}_{A'}^{(1)}$},\iota\raisebox{-1pt}{${}_{A'}^{(1)}$}\bigr) \!\rightarrow\!\bigl(A',\mu\raisebox{-1pt}{${}_{A'}^{(2)}$},\iota\raisebox{-1pt}{${}_{A'}^{(2)}$}\bigr)$} is an~isomorphism of $\mathcal{C}'$-algebras.
\end{proof}

Taking \smash{$\mu_A^{(1)}=\mu_A^{(2)}$} and \smash{$\mu_{A'}^{(1)}=\mu_{A'}^{(2)}$} in part (2) of the preceding theorem, we get the following.

\begin{Corollary}\label{cor:Aug-gps-iso}
In the setting of Theorem {\rm\ref{thm:C'-alg-to-C-alg}}, suppose $(A,\mu_A,\iota_A)$ is a commutative algebra in~$\mathcal{C}$ and $(A',\mu_{A'},\iota_{A'})$ is a commutative algebra in $\mathcal{C}'$ such that $(\mu_{A},\iota_{A})$ and $(\mu_{A'},\iota_{A'})$ are related by~\eqref{eqn:A'-A-structure-reln}. Then $\mathrm{Aut}_\mathcal{C}(A)\cong\mathrm{Aut}_{\mathcal{C}'}\bigl(A'\bigr)$, with isomorphism $g\mapsto g'$ given by \eqref{eqn:g'-g-reln}.
\end{Corollary}

Next, still in the setting of Theorem~\ref{thm:C'-alg-to-C-alg}, we consider the relation between simplicity of the $\mathcal{C}'$-algebra $A'$ and ideals of the $\mathcal{C}$-algebra $A$. Since $\Phi$ need not be a $\mathcal{C}$-isomorphism, $A$ need not be~simple if $A'$ is, but it will be almost simple under mild conditions. First we prove the following.
\begin{Lemma}\label{lem:im-phi+J-ideal}
In the setting of Theorem {\rm\ref{thm:C'-alg-to-C-alg}}, suppose $(A,\mu_A,\iota_A)$ is a commutative algebra in~$\mathcal{C}$ and $(A',\mu_{A'},\iota_{A'})$ is a commutative algebra in $\mathcal{C}'$ such that $(\mu_{A},\iota_{A})$ and $(\mu_{A'},\iota_{A'})$ are related by~\eqref{eqn:A'-A-structure-reln}, and assume that $A=\mathbf{1}\oplus J$ and $A'=\mathbf{1}'\oplus J$ for some object $J$ in $\mathcal{C}'$. Assume also that~$\iota_A$ and~$\iota_{A'}$ are the inclusions of~$\mathbf{1}$ and~$\mathbf{1}'$ into the direct sums $\mathbf{1}\oplus J$ and $\mathbf{1}'\oplus J$, respectively, and that $\Phi=\varphi\oplus\mathrm{Id}_J$. Then $\operatorname{Im} \Phi=\operatorname{Im} \varphi\oplus J$ is an ideal of $A$.
\end{Lemma}
\begin{proof}
Since $A=\mathbf{1}\oplus J$, we need to show $\operatorname{Im} \mu_A\vert_{\mathbf{1}\boxtimes\operatorname{Im} \Phi}\subseteq\operatorname{Im} \Phi$ and $\operatorname{Im} \mu_A\vert_{J\boxtimes\operatorname{Im} \Phi}\subseteq\operatorname{Im} \Phi$. The first inclusion holds because $\mu_A\vert_{\mathbf{1}\boxtimes\operatorname{Im} \Phi} =l_{\operatorname{Im} \Phi}$, and the second holds because
\begin{gather*}
\operatorname{Im} \mu_A\vert_{J\boxtimes\operatorname{Im} \Phi} =\operatorname{Im} \mu_A\circ(\Phi\boxtimes\Phi)\vert_{J\boxtimes A'} =\operatorname{Im} \Phi\circ\mu_{A'}\vert_{J\boxtimes A'},
\end{gather*}
using \eqref{eqn:A'-A-structure-reln}.
\end{proof}

\begin{Proposition}\label{prop:A'-simple-A-almost-simple}
In the setting of Lemma {\rm\ref{lem:im-phi+J-ideal}}, assume also that any subobject $I\subseteq\operatorname{Im} \Phi$ decomposes as a direct sum $I=I_{\mathbf{1}}\oplus I_J$ in $\mathcal{C}$ with $I_{\mathbf{1}}\subseteq\operatorname{Im} \varphi$ and $I_J\subseteq J$, and that any subobject~${I'\subseteq A'}$ decomposes as a direct sum $I'=I_{\mathbf{1}'}\oplus I_J$ in $\mathcal{C}'$ with $I_{\mathbf{1}'}\subseteq\mathbf{1}'$ and $I_J\subseteq J$.
\begin{enumerate}\itemsep=0pt
\item[$(1)$] If $A'$ is a simple commutative algebra in $\mathcal{C}'$ and $\operatorname{Im} \varphi$ is a simple object of $\mathcal{C}$, then $\operatorname{Im} \Phi$ is a simple ideal of $A$.
\item[$(2)$] Conversely, if $\operatorname{Im} \Phi$ is a simple ideal of $A$ and $\mathbf{1}'$ is a~simple object of $\mathcal{C}'$, then $A'$ is a~simple commutative algebra in $\mathcal{C}'$.
\end{enumerate}
\end{Proposition}
\begin{proof}
(1) To show that $\operatorname{Im} \Phi=\operatorname{Im} \varphi\oplus J$ is a simple ideal of $A$, let $I\subseteq\operatorname{Im} \Phi$ be any non-zero ideal. By assumption, $I=I_\mathbf{1}\oplus I_J$ for $\mathcal{C}$-subobjects $I_\mathbf{1}\subseteq\operatorname{Im} \varphi$ and $I_J\subseteq J$, with either $I_\mathbf{1}\neq 0$ or~${I_J\neq 0}$. In the first case, $I_\mathbf{1}=\operatorname{Im} \varphi$ since $\operatorname{Im} \varphi$ is simple, so
\begin{gather*}
I\supseteq \operatorname{Im} \mu_A\vert_{A\boxtimes\operatorname{Im} \varphi}\supseteq\operatorname{Im} \mu_A\circ(\Phi\boxtimes\Phi)\vert_{A'\boxtimes\mathbf{1}'} =\operatorname{Im} \Phi\circ\mu_{A'}\vert_{A'\boxtimes\mathbf{1}'} =\operatorname{Im} \Phi\circ r_{A'}' =\operatorname{Im} \Phi.
\end{gather*}
Thus $I=\operatorname{Im} \Phi$, proving that $\operatorname{Im} \Phi$ is a simple ideal if $I_\mathbf{1}\neq 0$.

In the second case, that $I_J\neq 0$, the ideal $\operatorname{Im} \mu_{A'}\vert_{A'\boxtimes I_J}$ of $A'$ generated by $I_J$ is equal to $A'$ because $A'$ is simple. Thus since $\operatorname{Im} \mu_{A'}\vert_{\mathbf{1}'\boxtimes I_J}=I_J\subseteq J$, and since $\operatorname{Im} \mu_{A'}\vert_{J\boxtimes I_J}$ decomposes as the direct sum of subobjects of $\mathbf{1}'$ and $J$ by assumption, we get $\mathbf{1}'\subseteq\operatorname{Im} \mu_{A'}\vert_{J\boxtimes I_J}$. Then
\begin{gather*}
I\supseteq\operatorname{Im} \mu_A\vert_{J\boxtimes I_J} =\operatorname{Im} \mu_A\circ(\Phi\boxtimes\Phi)\vert_{J\boxtimes I_J}=\operatorname{Im} \Phi\circ\mu_{A'}\vert_{J\boxtimes I_J}\supseteq\operatorname{Im} \Phi\vert_{\mathbf{1}'} =\operatorname{Im} \varphi.
\end{gather*}
Thus again $I=\operatorname{Im} \Phi$ by the argument of the preceding paragraph. This proves the first part of the proposition.

(2) Conversely, to show that $A'$ is a simple algebra, let $I'$ be a non-zero ideal of $A'$. By assumption, $I'=I_{\mathbf{1}'}\oplus I_J$ for some $I_{\mathbf{1}'}\subseteq\mathbf{1}'$ and $I_J\subseteq J$, with either $I_{\mathbf{1}'}\neq 0$ or $I_J\neq 0$. In the first case, $I_{\mathbf{1}'}=\mathbf{1}'$ since we assume $\mathbf{1}'$ is simple in $\mathcal{C}'$. Then $I'=A'$ because $\mu_{A'}\vert_{A'\boxtimes\mathbf{1}'}=r_{\mathbf{1}'}'$ is surjective. Thus $A'$ is simple if $I_{\mathbf{1}'}\neq 0$.

In the second case, that $I_J\neq 0$, the ideal $\operatorname{Im} \mu_A\vert_{A\boxtimes I_J}$ of $A$ generated by $I_J$ is equal to~$\operatorname{Im} \Phi$ because $\operatorname{Im} \Phi$ is a simple ideal. Thus since $\operatorname{Im} \mu_A\vert_{\mathbf{1}\boxtimes I_J}=I_J\subseteq J$ and since $\operatorname{Im} \mu_A\vert_{J\boxtimes I_J}$ decomposes as a direct sum of subobjects of $\mathbf{1}$ and $J$ by assumption, we get
\begin{gather*}
\operatorname{Im} \varphi\subseteq\operatorname{Im} \mu_A\vert_{J\boxtimes I_J}=\operatorname{Im} \mu_A\circ(\Phi\boxtimes\Phi)\vert_{J\boxtimes I_J} =\operatorname{Im} \Phi\circ\mu_{A'}\vert_{J\boxtimes I_J}.
\end{gather*}
Since $\operatorname{Im} \Phi\vert_J =J$, it follows that $\operatorname{Im} \mu_{A'}\vert_{J\boxtimes I_J}$ is not contained in $J$. Since by assumption $\operatorname{Im} \mu_{A'}\vert_{J\boxtimes I_J}$ decomposes as a direct sum of subobjects of $\mathbf{1}'$ and $J$, this forces $I_{\mathbf{1}'}\neq 0$, and we get $I=A'$ as in the preceding paragraph.
\end{proof}

We can now apply the preceding general results to the case $\mathcal{C}=\operatorname{Ind}(\mathcal{O}_{c_{p,q}})$ and $\mathcal{C}'=\operatorname{Ind}(\mathcal{C}_{\mathrm{PSL}_2})$, with $\mathbf{1}=\mathcal{K}_{1,1}$ and $\mathbf{1}'=\mathcal{K}_{1,1}'=\mathcal{L}_0$. We take the map $\varphi\colon \mathbf{1}'\rightarrow\mathbf{1}$ to be the one appearing in \eqref{eqn:ln'-rn'-def}, so that \eqref{eqn:l'-r'-gen-def} holds in this setting.
Also, $\operatorname{Im} \varphi\cong\mathcal{L}_{2p-1,1}$ is simple in $\mathcal{C}$, while~$\mathbf{1}'$ is simple in $\mathcal{C}'$, as required in Proposition \ref{prop:A'-simple-A-almost-simple}.
Recalling \eqref{eqn:Wpq} and \eqref{eqn:Wpq'-forget-decomp}, we set $A=\mathbf{1}\oplus J$ and $\mathbf{1}'\oplus J$, where $J=\bigoplus_{n=1}^\infty (2n+1)\cdot\mathcal{L}_{2n}$, and then we define $\Phi=\varphi\oplus\mathrm{Id}_J$, as required in Lemma~\ref{lem:im-phi+J-ideal}. Note that as required in Proposition \ref{prop:A'-simple-A-almost-simple}, any $\mathcal{C}$-subobject $I\subseteq\operatorname{Im} \Phi$ decomposes as a direct sum of subobjects of $\operatorname{Im} \varphi\cong\mathcal{L}_{2p-1,1}$ and $J$ because $\operatorname{Im} \Phi$ is semisimple and $\mathcal{L}_{2p-1,1}$ does not occur as a direct summand of $J$. Similarly, any $\mathcal{C}'$-subobject of $A'$ decomposes as a~direct sum of subobjects of $\mathbf{1}'=\mathcal{L}_0$ and $J$ because $\mathcal{L}_0$ does not occur as a direct summand of~$J$.
We~can now prove the main theorem of this paper.

\begin{Theorem}\label{thm1}
Let $p,q\in\mathbb{Z}_{\geq 2}$ be coprime.
\begin{enumerate}\itemsep=0pt
\item[$(1)$] There is a unique up to isomorphism VOA $W_{p,q}$ of central charge $c_{p,q}$ such that
\begin{gather*}
W_{p,q}\cong\mathcal{K}_{1,1}\oplus\bigoplus_{n=2}^\infty (2n-1)\cdot\mathcal{L}_{2np-1,1}
\end{gather*}
as a $\mathcal{V}{\rm ir}$-module and such that $I_{p,q}=\bigoplus_{n=1}^\infty (2n-1)\cdot\mathcal{L}_{2np-1,1}$ is a simple ideal.

\item[$(2)$] The VOA automorphism group of $W_{p,q}$ is isomorphic to $\mathrm{PSL}_2(\mathbb{C})$, and
\begin{gather}\label{eqn:Wpq-final-decomp}
W_{p,q}\cong (V_0\otimes\mathcal{K}_{1,1})\oplus\bigoplus_{n=2}^\infty V_{2n-2}\otimes\mathcal{L}_{2np-1,1}
\end{gather}
as a $\mathrm{PSL}_2(\mathbb{C})\times\mathcal{V}{\rm ir}$-module.
\end{enumerate}
\end{Theorem}
\begin{proof}
(1) By \cite{HKL}, and using the notation in the paragraph before the theorem statement, it is equivalent to show that there is a unique commutative algebra structure on $A=\mathbf{1}\oplus J$ such that $\operatorname{Im} \varphi\oplus J$ is a simple ideal. To apply the preceding results, we still need to check that the linear maps in \eqref{eqn:A0-iso} and \eqref{eqn:An-iso} are isomorphisms. For more uniform notation, we set $A^{\boxtimes 0}=\mathbf{1}$, \smash{$\bigl(A'\bigr)^{\boxtimes 0}=\mathbf{1}'$} and $\Phi^{\boxtimes 0} =\varphi$, so that we need to show
\begin{gather*}
\Phi\circ -\colon\ \operatorname{Hom}\bigl(\bigl(A'\bigr)^{\boxtimes n},A'\bigr)\longrightarrow\operatorname{Hom}\bigl(\bigl(A'\bigr)^{\boxtimes n},A\bigr),\\ -\circ\Phi^{\boxtimes n}\colon\ \operatorname{Hom}\bigl(A^{\boxtimes n},A\bigr)\longrightarrow\operatorname{Hom}\bigl(\bigl(A'\bigr)^{\boxtimes n},A\bigr)
\end{gather*}
are isomorphisms for $n=0,1,2,3$.

To show $\Phi\circ -$ is an isomorphism, we write
\begin{align*}
\bigl(A'\bigr)^{\boxtimes n} & \cong\bigoplus_{m_1,\ldots,m_n=0}^\infty (2m_1+1)\cdots (2m_n+1)\cdot\mathcal{L}_{2m_1}\boxtimes\cdots\boxtimes\mathcal{L}_{2m_n} \nonumber\\
&\cong\bigoplus_{m_1,\ldots,m_n=0}^\infty (2m_1+1)\cdots (2m_n+1)\cdot(N_{m_1,\ldots,m_n}\cdot\mathbf{1}'\oplus J_{m_1,\ldots,m_n}),
\end{align*}
where $N_{m_1,\ldots,m_n}$ is the multiplicity of $\mathbf{1}'$ in $\mathcal{L}_{2m_1}\boxtimes\cdots\mathcal{L}_{2m_n}$ and $J_{m_1,\ldots,m_n}$ is a finite direct sum of objects $\mathcal{L}_{2m}$ for $m\geq 1$. Under this identification, and observing that
\begin{gather*}
\operatorname{Hom}(\mathbf{1}', J) = \operatorname{Hom}(J_{m_1,\ldots,m_n},\mathbf{1}')=\operatorname{Hom}(J_{m_1,\ldots,m_n},\mathbf{1})=0,
\end{gather*}
the map $\Phi\circ -$ induces a map
\begin{gather*}
\prod_{m_1,\ldots,m_n=0}^\infty (2m_1+1)\cdots (2m_n+1)\cdot (N_{m_1,\ldots,m_n}\cdot\operatorname{Hom}(\mathbf{1}',\mathbf{1}')\oplus\operatorname{Hom}(J_{m_1,\ldots,m_n},J))\\
\qquad\longrightarrow\prod_{m_1,\ldots,m_n=0}^\infty (2m_1+1)\cdots (2m_n+1)\\
\phantom{\qquad\longrightarrow\prod_{m_1,\ldots,m_n=0}^\infty }{}\times (N_{m_1,\ldots,m_n}\cdot\operatorname{Hom}(\mathbf{1}',\mathbf{1})\oplus\operatorname{Hom}(J_{m_1,\ldots,m_n},J)),
\end{gather*}
which we need to show is an isomorphism. In fact, since $\Phi=\varphi\oplus\mathrm{Id}_J$, the induced map is
\begin{gather*}
\prod_{m_1,\ldots,m_n=0}^\infty (2m_1+1)\cdots (2m_n+1)\cdot (N_{m_1,\ldots,m_n}\cdot(\varphi\circ-)\oplus\mathrm{Id}_{\operatorname{Hom}(J_{m_1,\ldots,m_n},J)}),
\end{gather*}
and this is an isomorphism because $\varphi\circ -\colon \operatorname{Hom}(\mathbf{1}',\mathbf{1}')\rightarrow\operatorname{Hom}(\mathbf{1}',\mathbf{1})$ is: it sends the basis endomorphism $\mathrm{Id}_{\mathbf{1}'}$ to the basis homomorphism $\varphi$. This proves that $\Phi\circ -$ is an isomorphism for $n=0,1,2,3$ (and in fact for any $n\in\mathbb{Z}_{\geq 0}$).

To show $-\circ\Phi^{\boxtimes n}$ is an isomorphism, the direct sum decompositions $A=\mathbf{1}\oplus J$ and $A'=\mathbf{1}'\oplus J$ together with the unit isomorphisms of $\mathcal{C}$ and $\mathcal{C}'$ imply that
\begin{gather*}
\bigl(A'\bigr)^{\boxtimes n} \cong \mathbf{1}'\oplus\bigoplus_{m=1}^n \binom{n}{m}\cdot J^{\boxtimes m},\qquad A^{\boxtimes n}\cong\mathbf{1}\oplus\bigoplus_{m=1}^n \binom{n}{m}\cdot J^{\boxtimes m}.
\end{gather*}
Thus we need the map
\begin{gather*}
\operatorname{Hom}(\mathbf{1},A)\oplus\bigoplus_{m=1}^n \binom{n}{m}\cdot\operatorname{Hom}\bigl(J^{\boxtimes m}, A\bigr)\longrightarrow\operatorname{Hom}(\mathbf{1}',A)\oplus\bigoplus_{m=1}^n \binom{n}{m}\cdot\operatorname{Hom}\bigl(J^{\boxtimes m}, A\bigr)
\end{gather*}
induced by $-\circ\Phi^{\boxtimes n}$ to be an isomorphism. Using $\Phi=\varphi\oplus\mathrm{Id}_J$ and \eqref{eqn:l'-r'-gen-def}, it is straightforward to see that this induced map is
\begin{gather*}
(-\circ\varphi)\oplus\bigoplus_{m=1}^n \binom{n}{m}\cdot\mathrm{Id}_{\operatorname{Hom}(J^{\boxtimes m},A)}.
\end{gather*}
Now, $-\circ\varphi$ is an isomorphism because under the identifications
\begin{gather*}
\operatorname{Hom}(\mathbf{1},A)\cong\operatorname{Hom}(\mathbf{1},\mathbf{1}),\qquad\operatorname{Hom}(\mathbf{1}',A)\cong\operatorname{Hom}(\mathbf{1}',\mathbf{1}),
\end{gather*}
it sends the basis endomorphism $\mathrm{Id}_\mathbf{1}$ to the basis homomorphism $\varphi$. Thus $-\circ\Phi^{\boxtimes n}$ is an isomorphism for $n=0,1,2,3$ (and in fact for all $n\in\mathbb{Z}_{\geq 0}$).

We can now use Proposition \ref{prop2}, Theorem~\ref{thm:C'-alg-to-C-alg}\,(1), and Proposition \ref{prop:A'-simple-A-almost-simple}\,(1) to show that $A=\mathbf{1}\oplus J$ has a commutative $\mathcal{C}$-algebra structure such that $\operatorname{Im} \Phi=\operatorname{Im} \varphi\oplus J$ is a simple ideal. To show that this commutative algebra structure is unique up to isomorphism, suppose~\smash{$\bigl(A,\mu_A^{(1)},\iota_A^{(1)}\bigr)$} and~\smash{$\bigl(A,\mu_A^{(2)},\iota_A^{(2)}\bigr)$} are two commutative algebra structures such that $\operatorname{Im} \Phi$ is a~simple ideal. Then Proposition \ref{prop:A'-simple-A-almost-simple}\,(2) yields two simple commutative $\mathcal{C}'$-algebras \smash{$\bigl(A',\mu_{A'}^{(1)},\iota_{A'}^{(1)}\bigr)$} and \smash{$\bigl(A',\mu_{A'}^{(2)},\iota_{A'}^{(2)}\bigr)$} which must be isomorphic by Proposition \ref{prop2}. Thus the two commutative $\mathcal{C}$-algebra structures on $A$ are isomorphic by Theorem~\ref{thm:C'-alg-to-C-alg}\,(2).

(2) By Proposition \ref{prop2} and Corollary~\ref{cor:Aug-gps-iso}, the automorphism group of the commutative algebra structure on $A=\mathbf{1}\oplus J$ from the proof of part (1) of the theorem is isomorphic to $\mathrm{PSL}_2(\mathbb{C})$. Moreover, for $g'\in\mathrm{Aut}_{\mathcal{C}'}\bigl(A'\bigr)$, where $A'$ is the simple commutative algebra of Proposition \ref{prop2}, the identities \eqref{eqn:g'-g-reln} and $\Phi=\varphi\oplus\mathrm{Id}_J$ imply that the corresponding $g\in\mathrm{Aut}_\mathcal{C}(A)$ is given by~${g=\mathrm{Id}_\mathbf{1}\oplus g'\vert_J}$ (because $(\mathrm{Id}_\mathbf{1}\oplus g')\circ\Phi=\Phi\circ g'$). Thus by \eqref{eqn:Wpq'},
\begin{gather*}
A\cong (V_0\otimes\mathbf{1})\oplus\bigoplus_{n=1}^\infty V_{2n}\otimes\mathcal{L}_{2n}
\end{gather*}
as a $\mathrm{PSL}_2(\mathbb{C})$-module.

Finally, to complete the proof of the theorem, we just need to observe that the automorphisms of $A=W_{p,q}$ considered as a commutative algebra in $\mathcal{C}$ are the same as its automorphisms considered as a VOA. Indeed, by the isomorphism between commutative algebras in $\mathcal{C}$ and VOA extensions of $V_{c_{p,q}}$ proven in \cite{HKL}, elements of $\mathrm{Aut}_\mathcal{C}(A)$ are precisely the VOA automorphisms of~$W_{p,q}$ that fix $\mathbf{1}=V_{c_{p,q}}=\mathcal{K}_{1,1}$ pointwise. But all VOA automorphisms of $W_{p,q}$ fix $V_{c_{p,q}}$ pointwise because they fix the conformal vector $\omega$ which generates $V_{c_{p,q}}$ as a VOA. This completes the proof of the theorem.
\end{proof}

Since \cite[Theorem 5.4]{TW2} shows that the ideal $I_{p,q}$ of the triplet algebra $W_{p,q}$ introduced in \cite{FGST1-Log} is simple, Theorem~\ref{thm1} immediately implies the following.

\begin{Corollary}\label{cor:Wpq-aut-gp}
The triplet algebra $W_{p,q}$ introduced in {\rm\cite{FGST1-Log}} has automorphism group isomorphic to $\mathrm{PSL}_2(\mathbb{C})$, and the $\mathrm{PSL}_2(\mathbb{C})\times\mathcal{V}{\rm ir}$-module decomposition \eqref{eqn:Wpq-final-decomp} holds.
\end{Corollary}

\begin{Remark}
Note that there are now two independent ways to prove the existence of the VOA $W_{p,q}$ in Theorem~\ref{thm1}. The original method of \cite{FGST1-Log} defines $W_{p,q}$ as the intersection of the kernel of two screening operators on the lattice VOA~$V_{\sqrt{2pq}\mathbb{Z}}$ and then uses Felder complexes~\cite{Felder} and the socle series structure of Feigin--Fuchs modules~\cite{Feigin--Fuchs} to show that~$W_{p,q}$ has the correct decomposition as a~$\mathcal{V}{\rm ir}$-module. On the other hand, the method presented in this section defines~$W_{p,q}$ as the VOA structure corresponding (via~\cite{HKL}) to the commutative algebra obtained from $W_{p,q}'$ using the equations \eqref{eqn:A'-A-structure-reln}.
Using the first method, it is easier to see that~$W_{p,q}$ is a VOA, but it is more difficult to determine the structure of~$W_{p,q}$ as $\mathcal{V}{\rm ir}$-module, and it seems extremely difficult to rigorously determine the automorphism group of~$W_{p,q}$ using this definition. Using the second method, the main difficulty is that we need the technical fusion rule computation of Theorem~\ref{thm:Vir-fus-rules} to show that the $\mathrm{PSL}(2,\mathbb{C})\times\mathcal{V}{\rm ir}$-module $W_{p,q}'$ of \eqref{eqn:Wpq'} admits the structure of a simple commutative algebra in $\operatorname{Ind}(\operatorname{Rep} \mathrm{PSL}_2\otimes\mathcal{C}_{\mathrm{PSL}_2})$. But once we have this result, the determination of the automorphism group of $W_{p,q}$ in Theorem~\ref{thm1}\,(2) is fairly straightforward.
\end{Remark}

\section[Relations between Virasoro and triplet algebra representation theory]{Relations between Virasoro\\ and triplet algebra representation theory}\label{sec:Vir-trip-relations}

As in \cite[Section 7]{MY-cp1-Vir}, we can use Corollary~\ref{cor:Wpq-aut-gp} and VOA extension theory \cite{CKM1, CMY1,HKL}
to relate the Virasoro tensor category $\mathcal{O}_{c_{p,q}}$ studied in our previous paper \cite{MS-cpq-Vir}
with the representation category of the triplet algebra $W_{p,q}$.
Let $\operatorname{Rep}(W_{p,q})$ be the category of grading-restricted generalized $W_{p,q}$-modules, that is, modules with finite-dimensional generalized $L_0$-eigenspaces.
It is a finite abelian category and a braided tensor category \cite{Hu-C2}, but it is not rigid. On the other hand, contragredient dual modules give $\operatorname{Rep}(W_{p,q})$ the weaker duality structure of a ribbon Grothendieck--Verdier category \cite{ALSW}.

Now as discussed in the previous section $A:=W_{p,q}$ has the structure of a commutative algebra in the braided tensor category $\mathcal{C}:=\operatorname{Ind}(\mathcal{O}_{c_{p,q}})$, and $G:=\mathrm{PSL}_2(\mathbb{C})$ acts on $A$ by algebra automorphisms. Let $\mathrm{Mod}_\mathcal{C}(A)$ be the category of modules for the commutative algebra $A$ in $\mathcal{C}$. Specifically, an object of $\mathrm{Mod}_\mathcal{C}(A)$ is an object $X$ of $\mathcal{C}$ equipped with a morphism $\mu_X\colon A\boxtimes X\rightarrow X$ satisfying unitality and associativity
\begin{gather*}
\mu_X\circ(\iota_A\boxtimes\mathrm{Id}_X) = l_X,\qquad\mu_X\circ(\mathrm{Id}_X\boxtimes\mu_X)=\mu_X\circ(\mu_A\boxtimes\mathrm{Id}_X)\circ\mathcal{A}_{A,A,X}.
\end{gather*}
A morphism from $(X_1,\mu_{X_1})$ to $(X_2,\mu_{X_2})$ in $\mathrm{Mod}_\mathcal{C}(A)$ is a $\mathcal{C}$-morphism $f\colon X_1\rightarrow X_2$ such that~${
f\circ\mu_{X_1}=\mu_{X_2}\circ (\mathrm{Id}_A\boxtimes f)}$.
Let $\mathrm{Mod}^0_\mathcal{C}(A)$ be the category of local $A$-modules in $\mathcal{C}$, which consists of $A$-modules $(X,\mu_X)$ such that
\[
\mu_X\circ\mathcal{R}_{X,A}\circ\mathcal{R}_{A,X}=\mu_X.
\]
Then $\mathrm{Mod}_\mathcal{C}(A)$ is a tensor category, and its subcategory~$\mathrm{Mod}_\mathcal{C}^0(A)$ is a braided tensor category.

By \cite[Proposition 4.14 and Theorem 5.13]{TW2}, all simple objects of $\operatorname{Rep}(W_{p,q})$ are (possibly infinite) direct sums of simple $\mathcal{V}{\rm ir}$-modules in $\mathcal{O}_{c_{p,q}}$. Thus by the same argument as in the proof of \cite[Proposition 3.1.3]{CMY2}, every object of $\operatorname{Rep}(W_{p,q})$ is an object of $\operatorname{Ind}(\mathcal{O}_{c_{p,q}})$ when considered as a~$\mathcal{V}{\rm ir}$-module.
So by \cite[Theorem 3.4]{HKL} and \cite[Theorem 3.65]{CKM1}, $\operatorname{Rep}(W_{p,q})$ is a braided tensor subcategory of $\mathrm{Mod}^0_\mathcal{C}(A)$. It is a proper subcategory because, for example, the ind-category $\mathcal{C}$ contains infinite direct sums of $W_{p,q}$-modules, while $\operatorname{Rep}(W_{p,q})$ does not.

There is a monoidal tensor functor of induction $F\colon \mathcal{C}\rightarrow\mathrm{Mod}_\mathcal{C}(A)$ defined by
\begin{gather*}
F(W) = (A\boxtimes W,(\mu_A\boxtimes\mathrm{Id}_W)\circ\mathcal{A}_{A,A,W}),\qquad F(f)=\mathrm{Id}_A\boxtimes f
\end{gather*}
for objects $W$ and morphisms $f$ in $\mathcal{C}$. Induction is right exact because the tensor product $\boxtimes$ on $\mathcal{C}$ is right exact. Induction also satisfies Frobenius reciprocity, that is, there is a natural isomorphism
\begin{align}\label{eqn:Frob-rec}
\operatorname{Hom}_{\mathcal{C}}(W, X) \xrightarrow{ \sim } \operatorname{Hom}_{A}(F(W), (X,\mu_X)),\qquad
 f \longmapsto \mu_X\circ(\mathrm{Id}_A\boxtimes f)
\end{align}
for all objects $W$ in $\mathcal{C}$ and $(X,\mu_X)$ in $\mathrm{Mod}_\mathcal{C}(A)$.
\begin{Definition}
Let $\mathcal{O}_{c_{p,q}}^0$ be the full subcategory of $\mathcal{O}_{c_{p,q}}$ consisting of objects $W$ such that $F(W)$ is an object of $\operatorname{Rep}(W_{p,q})$.
\end{Definition}

Using \cite[Proposition 2.65]{CKM1}, for example, \smash{$\mathcal{O}_{c_{p,q}}^0$} is equivalently the full subcategory of objects~$W$ of $\mathcal{O}_{c_{p,q}}$ such that the double braiding $\mathcal{R}_{A,W}^2$ in $\mathcal{C}$ is the identity and such that $F(W)$ has finite-dimensional conformal weight spaces and a lower bound on conformal weights.
Because~$F$ is a monoidal functor, $\mathcal{O}_{c_{p,q}}^0$ is a monoidal subcategory of $\mathcal{O}_{c_{p,q}}$, and by
 \cite[Theorem~2.67]{CKM1} for example, \smash{$F\vert_{\mathcal{O}_{c_{p,q}}^0}\colon\mathcal{O}_{c_{p,q}}^0\rightarrow\operatorname{Rep}(W_{p,q})$} is a braided monoidal functor.

 \begin{Proposition}\label{prop:Ocpq0-first-properties}
 The braided monoidal category $\mathcal{O}_{c_{p,q}}^0$ is closed under quotients, contains all Kac modules $\mathcal{K}_{r,s}$ for $r,s\in\mathbb{Z}_{\geq 1}$, and contains all simple objects of $\mathcal{O}_{c_{p,q}}$.
 \end{Proposition}
 \begin{proof}
 First, $\mathcal{O}_{c_{p,q}}^0$ is closed under quotients because $\mathcal{O}_{c_{p,q}}$ and $\operatorname{Rep}(W_{p,q})$ are closed under quotients, and because $F$ is right exact and thus preserves surjections.

 Next, to show that $\mathcal{K}_{1,2}$ is an object of \smash{$\mathcal{O}_{c_{p,q}}^0$}, we note from \eqref{eqn:Wpq} and \cite[Theorem 6.8]{MS-cpq-Vir} that
 \begin{gather*}
 F(\mathcal{K}_{1,2})\cong W_{p,q}\boxtimes\mathcal{K}_{1,2}\cong\mathcal{K}_{1,2}\oplus\bigoplus_{n=2}^\infty (2n-1)\cdot\mathcal{L}_{1,2nq-2}
 \end{gather*}
 as a $\mathcal{V}{\rm ir}$-module. Since the conformal weights satisfy
 \begin{gather*}
 h_{1,2nq-2} = pqn^2-(2p+q)n+\frac{3}{2}+\frac{3p}{4q} =h_{1,2}+(np-1)(nq-2),
 \end{gather*}
 the conformal weight spaces of $F(\mathcal{K}_{1,2})$ are finite dimensional, and the conformal weights of $F(\mathcal{K}_{1,2})$ have a lower bound and are all congruent to $h_{1,2}$ mod $\mathbb{Z}$. The balancing equation for the double braiding then implies
 \begin{gather*}
 \mathcal{R}^2_{W_{p,q},\mathcal{K}_{1,2}} = {\rm e}^{2\pi{\rm i }L_0}\circ\bigl({\rm e}^{-2\pi {\rm i} L_0}\boxtimes{\rm e}^{-2\pi{\rm i } L_0}\bigr) = {\rm e}^{2\pi {\rm i }(h_{1,2}-0-h_{1,2})} =\mathrm{Id}_{W_{p,q},\mathcal{K}_{1,2}},
 \end{gather*}
 so $\mathcal{K}_{1,2}$ is an object of $\mathcal{O}_{c_{p,q}}^0$. By $c_{p,q}=c_{q,p}$ symmetry, $\mathcal{K}_{2,1}$ is also an object of $\mathcal{O}_{c_{p,q}}^0$.

 For the remaining Kac modules in $\mathcal{O}_{c_{p,q}}$, it follows from \cite[Theorem 4.7]{MS-cpq-Vir} that for $r,s\in\mathbb{Z}_{\geq 1}$, there is a surjection \smash{$\mathcal{K}_{2,1}^{\boxtimes(r-1)}\boxtimes\mathcal{K}_{1,2}^{(s-1)}\twoheadrightarrow \mathcal{K}_{r,s}$}. So because \smash{$\mathcal{O}_{c_{p,q}}^0$} is a monoidal subcategory of~\smash{$\mathcal{O}_{c_{p,q}}$} which is closed under quotients, $\mathcal{K}_{r,s}$ is an object of~\smash{$\mathcal{O}_{c_{p,q}}^0$} for all $r,s\in\mathbb{Z}_{\geq 1}$. Finally, the simple objects of $\mathcal{O}_{c_{p,q}}$ can be parametrized by $\mathcal{L}_{r,s}$ for $1\leq r\leq p$ and $s\in\mathbb{Z}_{\geq 1}$ such that $ps\geq qr$. For all such $(r,s)$, the diagrams in \cite[Section 2.3]{MS-cpq-Vir} show that $\mathcal{L}_{r,s}$ is a quotient of $\mathcal{K}_{r,s}$, so all $\mathcal{L}_{r,s}$ are objects of $\mathcal{O}_{c_{p,q}}^0$.
 \end{proof}

Because $G=\mathrm{PSL}_2(\mathbb{C})$ acts on $W_{p,q}$ by automorphisms, the category $\mathrm{Mod}_\mathcal{C}(A)$ has a $G$-equivariantization $\mathrm{Mod}_\mathcal{C}(A)$ as defined in \cite[Section 2.7]{EGNO} for example. Concretely, this is the category consisting of objects $(X,\mu_X)$ of $\mathrm{Mod}_\mathcal{C}(A)$ equipped with a continuous representation~${\varphi_X\colon G\rightarrow\mathrm{Aut}_\mathcal{C}(X)}$ such that
$
\varphi_X(g)\circ\mu_X=\mu_X\circ(g\boxtimes\varphi_X(g))
$
for all $g\in G$. Morphisms from $(X_1,\mu_{X_1},\varphi_{X_1})$ to $(X_2,\mu_{X_2},\varphi_{X_2})$ in $\mathrm{Mod}_\mathcal{C}(A)^G$ consist of all morphisms $f\colon X_1\rightarrow X_2$ in $\mathrm{Mod}_\mathcal{C}(A)$ such that
$
\varphi_{X_2}(g)\circ f =f\circ\varphi_{X_1}(g)
$
for all $g\in G$. Note that the $A$-action $\mu_X$ of an object of $\mathrm{Mod}_\mathcal{C}(A)$ is equivalent to a $V_{c_{p,q}}$-module intertwining operator $Y_X\colon W_{p,q}\otimes X\rightarrow X[\log x]\lbrace x\rbrace$ such that $Y_X=\mu_X\circ\mathcal{Y}_\boxtimes$, where $\mathcal{Y}_\boxtimes$ is the canonical tensor product $V_{c_{p,q}}$-module intertwining operator of type \smash{$\binom{X}{W_{p,q} X}$} in $\mathcal{C}$.
Thus in vertex algebraic terms, the $G$-representation $\varphi_X$ of an object of $\mathrm{Mod}_\mathcal{C}(A)$ satisfies
\begin{gather}\label{eqn:concrete-G-equiv}
\varphi_X(g)(Y_X(a,x)b) = Y_X(g\cdot a,x)\varphi_X(b)
\end{gather}
for all $g\in G$, $a\in W_{p,q}$, and $b\in X$.

The $G$-equivariantization $\mathrm{Mod}_\mathcal{C}(A)^G$ is a braided tensor category (see, for example, \cite{EGNO} or the discussion in \cite[Section 7]{MY-cp1-Vir}),
and as in \cite[Lemma 7.13]{MY-cp1-Vir},
induction defines a braided monoidal functor $F\colon\mathcal{O}_{c_{p,q}}\rightarrow\mathrm{Mod}_\mathcal{C}(A)^G$ (though we cannot say that $F$ is exact because $\mathcal{O}_{c_{p,q}}$ is not rigid). For an object $W$ in $\mathcal{O}_{c_{p,q}}$, the representation $\varphi_{F(W)}$ is defined by $\varphi_{F(W)}(g)=g\boxtimes\mathrm{Id}_W$ for all~${g\in G}$. Note that the subcategory $\operatorname{Rep}(W_{p,q})\subseteq\mathrm{Mod}_\mathcal{C}(A)$ also has a $G$-equivariantization which is a braided tensor subcategory of $\mathrm{Mod}_\mathcal{C}(A)^G$, and thus $F$ restricts to a braided tensor functor from $\mathcal{O}_{c_{p,q}}^0$ to $\operatorname{Rep}(W_{p,q})^{\mathrm{PSL}_2(\mathbb{C})}$.
\begin{Theorem}
The braided monoidal induction functor $F\colon\mathcal{C}\rightarrow\mathrm{Mod}_\mathcal{C}(A)^G$ is fully faithful. In particular, \smash{$F\vert_{\mathcal{O}_{c_{p,q}}^0}\colon\mathcal{O}_{c_{p,q}}^0\rightarrow\operatorname{Rep}(W_{p,q})^{\mathrm{PSL}_2(\mathbb{C})}$} is fully faithful.
\end{Theorem}
\begin{proof}
We continue to use the notation $A$ for $W_{p,q}$ considered as a commutative algebra in~${\mathcal{C}=\operatorname{Ind}(\mathcal{O}_{c_{p,q}})}$, and we continue to set $G=\mathrm{PSL}_2(\mathbb{C})$. We need to show that for any objects $W_1$, $W_2$ in $\mathcal{C}$, the map
\begin{align*}
F\colon\ \operatorname{Hom}_\mathcal{C}(W_1,W_2) \longrightarrow\operatorname{Hom}_{G\times A}(F(W_1),F(W_2)),\qquad
 f \longmapsto \mathrm{Id}_A\boxtimes f
\end{align*}
is an isomorphism. Since $A=\mathbf{1}\oplus J$ as an object of $\mathcal{C}$, where $\mathbf{1}=\mathcal{K}_{1,1}$ and $J=\bigoplus_{n=1}^\infty V_{2n}\otimes\mathcal{L}_{2n}$ as a $G\times\mathcal{V}{\rm ir}$-module, we can define a projection $\pi_A\colon A\rightarrow\mathbf{1}$ such that $\pi_A\circ\iota_A=\mathrm{Id}_\mathbf{1}$. This allows us to define
\begin{gather*}
\widetilde{F}\colon\ \operatorname{Hom}_{G\times A}(F(W_1),F(W_2))\longrightarrow \operatorname{Hom}_\mathcal{C}(W_1,W_2)
\end{gather*}
such that for $\Gamma\in\operatorname{Hom}_{G\times A}(F(W_1),F(W_2))$, $\widetilde{F}(\Gamma)$ is the composition
\begin{gather*}
W_1\xrightarrow{l_{W_1}^{-1}} \mathbf{1}\boxtimes W_1\xrightarrow{\iota_A\boxtimes\mathrm{Id}_{W_1}} A\boxtimes W_1\xrightarrow{\Gamma} A\boxtimes W_2\xrightarrow{\pi_A\boxtimes\mathrm{Id}_{W_2}} \mathbf{1}\boxtimes W_2\xrightarrow{l_{W_2}} W_2.
\end{gather*}
Since $\pi_A\circ\iota_A=\mathrm{Id}_{\mathbf{1}}$, it is clear that $\widetilde{F}(F(f))=f$ for all $\mathcal{C}$-morphisms $f\colon W_1\rightarrow W_2$.

To show that $F\bigl(\widetilde{F}(\Gamma)\bigr)=\Gamma$ as well, note that if $\Gamma\colon F(W_1)\rightarrow F(W_2)$ is a morphism in~$\mathrm{Mod}_\mathcal{C}(A)^G$, then for all $g\in G$,
\begin{align*}
\varphi_{F(W_2)}(g)\circ \Gamma\circ(\iota_A\boxtimes\mathrm{Id}_{W_1}) &=\Gamma\circ\varphi_{F(W_1)}(g)\circ(\iota_A\boxtimes\mathrm{Id}_{W_1})\nonumber\\
& =\Gamma\circ(g\boxtimes\mathrm{Id}_{W_1})\circ(\iota_A\boxtimes\mathrm{Id}_{W_1})=\Gamma\circ(\iota_A\boxtimes\mathrm{Id}_{W_1}).
\end{align*}
That is, $\operatorname{Im} \Gamma\circ(\iota_A\boxtimes\mathrm{Id}_{W_1})$ is contained in the subspace of $G$-invariants of $F(W_2)$. In fact, this subspace of $G$-invariants is precisely $\operatorname{Im} \iota_A\boxtimes\mathrm{Id}_{W_2}$ since $J\boxtimes W_2\cong\bigoplus_{n=1}^\infty V_{2n}\otimes(\mathcal{L}_{2n}\boxtimes W_2)$ is a~direct sum of non-trivial simple $G$-modules. It follows that
\begin{gather*}
\Gamma\circ(\iota_A\boxtimes\mathrm{Id}_{W_1}) =(\iota_A\boxtimes\mathrm{Id}_{W_2})\circ(\pi_A\boxtimes\mathrm{Id}_{W_2})\circ\Gamma\circ(\iota_A\boxtimes\mathrm{Id}_{W_2}).
\end{gather*}
Using this identity along with the triangle axiom of a tensor category the right unit property of the algebra multiplication $\mu_A$, and naturality of the associativity isomorphisms the morphism~$F\bigl(\widetilde{F}(\Gamma)\bigr)=\mathrm{Id}_A\boxtimes \widetilde{F}(\Gamma)$ reduces to the composition
\begin{align*}
A\boxtimes W_1 & \xrightarrow{\mathrm{Id}_A\boxtimes l_{W_1}^{-1}} A\boxtimes(\mathbf{1}\boxtimes W_1)\xrightarrow{\mathrm{Id}_A\boxtimes(\iota_A\boxtimes\mathrm{Id}_{W_1})} A\boxtimes(A\boxtimes W_1)\nonumber\\
&\xrightarrow{\mathrm{Id}_A\boxtimes\Gamma} A\boxtimes(A\boxtimes W_2)\xrightarrow{\mathcal{A}_{A,A,W_2}} (A\boxtimes A)\boxtimes W_2\xrightarrow{\mu_A\boxtimes\mathrm{Id}_{W_2}} A\boxtimes W_2.
\end{align*}
Since $\Gamma$ is in particular a morphism in $\mathrm{Mod}_\mathcal{C}(A)$, this composition equals
\begin{align*}
A\boxtimes W_1 & \xrightarrow{\mathrm{Id}_A\boxtimes l_{W_1}^{-1}} A\boxtimes(\mathbf{1}\boxtimes W_1)\xrightarrow{\mathrm{Id}_A\boxtimes(\iota_A\boxtimes\mathrm{Id}_{W_1})} A\boxtimes(A\boxtimes W_1)\nonumber\\
&\xrightarrow{\mathcal{A}_{A,A,W_1}} (A\boxtimes A)\boxtimes W_1\xrightarrow{\mu_A\boxtimes\mathrm{Id}_{W_1}} A\boxtimes W_1 \xrightarrow{\Gamma} A\boxtimes W_2.
\end{align*}
But this is just $\Gamma$ by naturality of the associativity isomorphisms, the triangle axiom, and the right unit property of $\mu_A$. This completes the proof that the induction functor $F$ is an isomorphism on morphisms.
\end{proof}

Unlike in \cite[Theorem 7.14]{MY-cp1-Vir}, where it was shown that the $\mathrm{PSL}_2(\mathbb{C})$-equivariantization of the representation category of the $W_{p,1}$ triplet VOA is braided tensor equivalent to a representation category of the Virasoro algebra at central charge $c_{p,1}$, it is \textit{not} true that induction at $c_{p,q}$ central charge for coprime $p,q\geq 2$ gives an equivalence between $\mathcal{O}_{c_{p,q}}^0$ and $\operatorname{Rep}(W_{p,q})^{\mathrm{PSL}_2(\mathbb{C})}$. The problem is that unlike in the $q=1$ case, $W_{p,q}$ is not a simple VOA but rather has the simple Virasoro VOA $L_{c_{p,q}}$ as its non-trivial simple quotient. This allows us to construct objects of~$\operatorname{Rep}(W_{p,q})^{\mathrm{PSL}_2(\mathbb{C})}$ which are not in the essential image of the induction functor.

\begin{Example}
 Let $X=V\otimes W$ where $V$ is any finite-dimensional continuous $G=\mathrm{PSL}_2(\mathbb{C})$-module and $(W,Y_W)$ is any $L_{c_{p,q}}$-module with finite-dimensional $L_0$-eigenspaces. Then $X$ admits the $G$-representation $\varphi_X(g)=g\otimes\mathrm{Id}_W$, and $X$ is also a $W_{p,q}$-module with vertex operator~$Y_X =\mathrm{Id}_V\otimes Y_W(\pi(-),x)$, where $\pi\colon W_{p,q}\rightarrow L_{c_{p,q}}$ is the surjective VOA homomorphism. Further, \eqref{eqn:concrete-G-equiv} holds because $\pi(g\cdot a)=\pi(a)$ for all $a\in W_{p,q}$ and $g\in \mathrm{PSL}_2(\mathbb{C})$, so $(X,Y_X,\varphi_X)$ is an object of~$\operatorname{Rep}(W_{p,q})^{\mathrm{PSL}_2(\mathbb{C})}$ However, if $V$ does not contain the trivial $G$-module $V_0$ as a direct summand, then $X$ is not in the essential image of the induction functor $F$, because any non-zero induced module always contains a non-zero $\mathrm{PSL}_2(\mathbb{C})$-invariant subspace.
 \end{Example}

Our next result will show that the above examples are typical. We continue to use the notation $A$ for $W_{p,q}$ considered as a commutative algebra in $\mathcal{C}=\operatorname{Ind}(\mathcal{O}_{c_{p,q}})$, and we continue to set $G=\mathrm{PSL}_2(\mathbb{C})$. Since $G$ acts continuously on any object $X$ of $\mathrm{Mod}_\mathcal{C}(A)$, we have $X=\bigoplus_{n=0}^\infty V_{2n}\otimes X_{2n}$, where $X_{2n}=\operatorname{Hom}_G(V_{2n},X)$ is a $\mathcal{V}{\rm ir}$-module which is an object of $\mathcal{C}$. We~use~$X^G$ to denote the set of $G$-invariants in $X$, that is, $X^G=V_0\otimes X_0$.

\begin{Lemma}\label{lem:no-G-inv}
If $(X,Y_X,\varphi_X)$ is an object of $\mathrm{Mod}_\mathcal{C}(A)^G$ such that $X^G=0$, then as a $G\times\mathcal{V}{\rm ir}$-module, $X =\bigoplus_{n=1}^\infty V_{2n}\otimes X_{2n}$ where each $X_{2n}$ is an $L_{c_{p,q}}$-module.
\end{Lemma}

\begin{proof}
By assumption, $X=\bigoplus_{n=1}^\infty V_{2n}\otimes X_{2n}$ where each $X_{2n}$ is an object of $\mathcal{C}$. To show that each $X_{2n}$ is actually an $L_{c_{p,q}}$-module, recall the simple ideal $I_{p,q}=\mathcal{L}_{2p-1,1}\oplus\bigoplus_{n=1}^\infty V_{2n}\otimes\mathcal{L}_{2n}$ from \eqref{eqn:Wpq-exact-seq}.

For any $n\in\mathbb{Z}_{\geq 1}$, there is nothing to prove if $X_{2n}=0$, so we assume $X_{2n}\neq 0$ and take an arbitrary non-zero $b\in X_{2n}$. Then $V_{2n}\otimes b$ is a non-zero $G$-submodule of $X$ with basis~\smash{$\big\lbrace v^{(i)}\otimes b\big\rbrace_{i=1}^{2n+1}$}, where \smash{$\big\lbrace v^{(i)}\big\rbrace_{i=1}^{2n+1}$} is an orthonormal basis with respect to a non-degenerate $\mathfrak{sl}_2$-invariant bilinear form on $V_{2n}$. Now take some non-zero $a\in\mathcal{L}_{2n}$, so that $V_{2n}\otimes a$ is a non-zero $G$-submodule of~$I_{p,q}$ with basis \smash{$\big\lbrace v^{(i)}\otimes a\big\rbrace_{i=1}^{2n+1}$}. Then
\begin{gather*}
\sum_{i=1}^{2n+1} Y_X\bigl(v^{(i)}\otimes a, x\bigr)\bigl(v^{(i)}\otimes b\bigr)\in X^G[\log x]\lbrace x \rbrace=0,
\end{gather*}
so because $I_{p,q}$ is a simple $W_{p,q}$-module, the argument in the proof of \cite[Lemma 4.18]{McR-G-equiv} using the analytic associativity of $Y_X$ and the Jacobson Density Theorem (see also \cite[Lemma 3.1]{DM}) shows that $Y_X\bigl(v^{(i)}\otimes a,x\bigr)\bigl(v^{(i)}\otimes b\bigr)=0$ for all $i$. In particular, the annihilator ideal
\begin{gather*}
\mathrm{Ann}_{W_{p,q}}\bigl(v^{(i)}\otimes b\bigr) =\big\lbrace w\in W_{p,q}\mid Y_X(w,x)\bigl(v^{(i)}\otimes b\bigr) = 0\big\rbrace
\end{gather*}
is non-zero and thus contains $I_{p,q}$ for all $i$ and all $b\in\mathcal{L}_{2n}$.

Since $X$ is spanned by the vectors $v^{(i)}\otimes b$ for non-zero $b\in X_{2n}$, $n\geq 1$, and $1\leq i\leq 2n+1$, we have now shown that $Y_X(a,x)b=0$ for all $a\in I_{p,q}$ and $b\in X$. Thus $(X,\overline{Y}_X)$ is a well-defined $L_{c_{p,q}}$-module, where $\overline{Y}_X=Y_X(\pi(-),x)$ and $\pi\colon W_{p,q}\rightarrow L_{c_{p,q}}$ is the quotient VOA homomorphism. In particular, each $X_{2n}$ is a (not necessarily grading restricted) $L_{c_{p,q}}$-module.
\end{proof}

As a first application, Lemma \ref{lem:no-G-inv} gives some information about the relation between arbitrary objects of $\mathrm{Mod}_\mathcal{C}(A)^G$ and the essential image of the induction functor.
Indeed, by Frobenius reciprocity \eqref{eqn:Frob-rec}, the inclusion $\iota_X\colon X^G\hookrightarrow X$ induces the $\mathrm{Mod}_\mathcal{C}(A)$-morphism
\begin{gather*}
f_X = \mu_X\circ(\mathrm{Id}_A\boxtimes \iota_X)\colon F\bigl(X^G\bigr)\longrightarrow X,
\end{gather*}
and $f_X$ is further a morphism in $\mathrm{Mod}_\mathcal{C}(A)^G$ because
\begin{align*}
\varphi_X(g)\circ f_X & =\varphi_X(g)\circ\mu_X\circ(\mathrm{Id}_A\boxtimes\iota_X) =\mu_X\circ(g\boxtimes\varphi_X(g))\circ(\mathrm{Id}_A\boxtimes\iota_X)\\
& =\mu_X\circ(\mathrm{Id}_A\boxtimes\iota_X)\circ(g\boxtimes\mathrm{Id}_{X^G})=f_X\circ\varphi_{F(X^G)}(g)
\end{align*}
for all $g\in G$. Although we cannot say that $f_X$ is an isomorphism in general (in contrast with \cite[Theorem 7.14]{MY-cp1-Vir} on the relation between the Virasoro and triplet algebras at $c_{p,1}$ central charge), the restriction of $f_X$ to $G$-invariants is an isomorphism. So the kernel and cokernel of $f_X$ have no $G$-invariants, and Lemma \ref{lem:no-G-inv} immediately yields the following.

\begin{Corollary}
For any object $(X,Y_X;\varphi_X)$ of $\mathrm{Mod}_\mathcal{C}(A)^G$, there is a $\mathrm{PSL}_2(\mathbb{C})\times W_{p,q}$-module exact sequence
\begin{gather*}
0\longrightarrow L\longrightarrow F\bigl(X^G\bigr)\xrightarrow{ f_X } X\longrightarrow\widetilde{L}\longrightarrow 0,
\end{gather*}
where $L$ and $\widetilde{L}$ are $\mathrm{PSL}_2(\mathbb{C})\times L_{c_{p,q}}$-modules on which $W_{p,q}$ acts through the quotient map $W_{p,q}\twoheadrightarrow L_{c_{p,q}}$.
\end{Corollary}

We can also use Lemma \ref{lem:no-G-inv} to show that the subcategory $\mathcal{O}_{c_{p,q}}^0\subseteq\mathcal{O}_{c_{p,q}}$ defined previously is closed under contragredient modules:
\begin{Proposition}\label{prop:Ocpq0-contragredients}
 If $W$ is an object of $\mathcal{O}_{c_{p,q}}^0$, then so is its contragredient $W'$.
\end{Proposition}
\begin{proof}
By assumption, the induced module $F(W)=\bigoplus_{n=0}^\infty V_{2n}\otimes(\mathcal{L}_{2n}\boxtimes W)$ is an object of the category $\operatorname{Rep}(W_{p,q})$ of grading-restricted generalized $W_{p,q}$-modules. We need to show the same for $F(W')$. First, since $W'$ has finite length and $F$ is right exact, induction on the length shows that $F(W')$ has finite-dimensional conformal weight spaces and a lower bound on conformal weights, provided the same holds for $F(L)$ whenever $L$ is a simple object of $\mathcal{O}_{c_{p,q}}$. Indeed, this holds for $F(L)$ by Proposition \ref{prop:Ocpq0-first-properties}. It remains to show that $\mu_{F(W')}\circ\mathcal{R}_{W_{p,q},F(W')}^2=\mu_{F(W')}$, or equivalently, the vertex operator
\begin{gather*}
Y_{F(W')} =\mu_{F(W')}\circ\mathcal{Y}_\boxtimes \colon\ W_{p,q}\otimes F(W')\longrightarrow F(W')[\log x]\lbrace x\rbrace
\end{gather*}
involves only integral powers of the formal variable $x$.

 To do so, we consider the $W_{p,q}$-module contragredient $F(W)'$, which like $F(W)$ is an object of $\operatorname{Rep}(W_{p,q})$, with vertex operator defined by
\begin{gather*}
\langle Y_{F(W)'}(v,x)w', w\rangle =\big\langle w', Y_{F(W)}\bigl(e^{xL_1}\bigl(-x^{-2}\bigr)^{L_0} v, x^{-1}\bigr)w\big\rangle
\end{gather*}
for $v\in W_{p,q}$, $w'\in F(W)'$, and $w\in F(W)$ \cite{FHL}. Using this relation and \eqref{eqn:concrete-G-equiv}, it is easy to see that $F(W)'$ is also an object of $\operatorname{Rep}(W_{p,q})^{\mathrm{PSL}_2(\mathbb{C})}$, with
\begin{gather*}
\langle \varphi_{F(W)'}(g)\cdot w', w\rangle = \big\langle w',\varphi_{F(W)}(g)^{-1}\cdot w\big\rangle
\end{gather*}
for all $g\in \mathrm{PSL}_2(\mathbb{C})$, $w'\in F(W)'$, and $w\in F(W)$. In particular, $F(W)'$ is the (graded) dual of~$F(W)$ as both $\mathrm{PSL}_2(\mathbb{C})$- and $\mathcal{V}{\rm ir}$-module, that is,{\samepage
\begin{gather}\label{eqn:F(W)'-decomp}
F(W)'\cong\bigoplus_{n=0}^\infty V_{2n}^*\otimes(\mathcal{L}_{2n}\boxtimes W)'\cong W'\oplus\bigoplus_{n=1}^\infty V_{2n}\otimes(\mathcal{L}_{2n}\boxtimes W)'
\end{gather}
as a $\mathrm{PSL}_2(\mathbb{C})\times\mathcal{V}{\rm ir}$-module.}

The decomposition \eqref{eqn:F(W)'-decomp} shows that $F(W')^{\mathrm{PSL}_2(\mathbb{C})}\cong (F(W)')^{\mathrm{PSL}_2(\mathbb{C})}\cong W'$, so we get a~$\mathrm{PSL}_2(\mathbb{C})\times W_{p,q}$-module homomorphism
\begin{gather*}
f\colon\ F(W') \xrightarrow{ \sim } F\bigl((F(W)')^{\mathrm{PSL}_2(\mathbb{C})}\bigr) \xrightarrow{f_{F(W)'}} F(W)',
\end{gather*}
which is an isomorphism on $\mathrm{PSL}_2(\mathbb{C})$-invariant subspaces.
Thus we have an exact sequence
\begin{gather}\label{eqn:F(W')-exact-seq}
0 \longrightarrow \mathrm{Ker} f \longrightarrow F(W') \longrightarrow \operatorname{Im} f \longrightarrow 0,
\end{gather}
where $\operatorname{Im} f$ is an object of $\operatorname{Rep}(W_{p,q})$ since it is a submodule of $F(W)'$, and $\mathrm{Ker} f$ has no $\mathrm{PSL}_2(\mathbb{C})$-invariants. Thus by Lemma \ref{lem:no-G-inv}, $\mathrm{Ker} f$ is an $L_{c_{p,q}}$-module, that is, a direct sum of modules $\mathcal{L}_{r,s}$ for $1\leq r\leq p-1$, $1\leq s\leq q-1$, and $W_{p,q}$ acts on $\mathrm{Ker} f$ through the quotient map $W_{p,q}\twoheadrightarrow L_{c_{p,q}}$. Note that $\mathrm{Ker} f$ is a grading-restricted $L_{c_{p,q}}$-module since $F(W')$ has finite-dimensional conformal weight spaces.

We now write the vertex operator \smash{$Y_{F(W')}$} as
\begin{gather*}
Y_{F(W')} =\sum_{\lambda+\mathbb{Z}\in\mathbb{C}/\mathbb{Z}}\sum_{k=0}^K \mathcal{Y}_{\lambda,k} x^\lambda (\log x)^k,
\end{gather*}
where $\mathcal{Y}_{\lambda,k}\colon W_{p,q}\otimes F(W')\longrightarrow F(W')((x))$ is a Laurent series and $K$ is related to the maximum Jordan block size of $L_0$ acting on $F(W')$ (see \cite[Proposition 3.20\,(c)]{HLZ2}). Note that $K$ is finite by~\eqref{eqn:F(W')-exact-seq}, since $L_0$ acts semisimply on $\mathrm{Ker} f$ and the maximum Jordan block size for $L_0$ acting on the object $\operatorname{Im} f$ of $\operatorname{Rep}(W_{p,q})$ is finite. For any $\lambda+\mathbb{Z}\in\mathbb{C}/\mathbb{Z}$, $x^\lambda\mathcal{Y}_{\lambda,K}$ is a $V_{c_{p,q}}$-module intertwining operator of type \smash{$\binom{F(W')}{W_{p,q} F(W')}$}. It suffices to show that $\mathcal{Y}_{\lambda,K}=0$ if either $\lambda\notin\mathbb{Z}$ or~${K>0}$.

Indeed, if $\lambda\notin\mathbb{Z}$ or $K>0$, then $\mathcal{Y}_{\lambda,K}(a,x)b\in(\mathrm{Ker} f)((x))$ for all $a\in W_{p,q}$, $b \in F(W')$ because $\operatorname{Im} f\subseteq F(W)'$ is local and hence
\begin{align*}
\sum_{\lambda+\mathbb{Z}\in\mathbb{C}/\mathbb{Z}}\sum_{k=0}^K f(\mathcal{Y}_{\lambda,K}(a,x)b) x^\lambda(\log x)^k & =f(Y_{F(W')}(a,x)b)\\
&=Y_{F(W)'}(a,x)f(b)\in F(W)'((x)).
\end{align*}
Thus $x^\lambda\mathcal{Y}_{\lambda,K}$ is an intertwining operator of type \smash{$\binom{\mathrm{Ker} f}{W_{p,q} F(W')}$}. Moreover, $\mathcal{Y}_{\lambda,K}\vert_{\mathcal{K}_{1,1}\otimes F(W')}=0$ since as a $V_{c_{p,q}}$-module, $F(W')$ is local with vertex operator $Y_{F(W')}\vert_{\mathcal{K}_{1,1}\otimes F(W')}$. Since $W_{p,q}=\mathcal{K}_{1,1}\oplus \bigoplus_{n=1}^\infty V_{2n}\otimes\mathcal{L}_{2n}$, it remains to show that $\mathcal{Y}_{\lambda,K}\vert_{(V_{2n}\otimes\mathcal{L}_{2n})\otimes F(W')}=0$. In fact, by symmetries of intertwining operators, $x^\lambda\mathcal{Y}_{\lambda,K}$ induces a $V_{c_{p,q}}$-module intertwining operator
\begin{gather*}
\mathcal{Y}\colon\ (\mathrm{Ker} f)'\otimes F(W')\longrightarrow(V_{2n}\otimes\mathcal{L}_{2n})((x)),
\end{gather*}
which equals $0$ if and only if $\mathcal{Y}_{\lambda,K}\vert_{(V_{2n}\otimes\mathcal{L}_{2n})\otimes F(W')}=0$. By \cite[Lemma 5.11]{MS-cpq-Vir}, the image of $\mathcal{Y}$ is an $L_{c_{p,q}}$-module. Thus because $\mathcal{L}_{2n}=\mathcal{L}_{2(n+1)p-1,1}$ is not an $L_{c_{p,q}}$-module, $\mathcal{Y}=0$ and then so is~${\mathcal{Y}_{\lambda,K}\vert_{(V_{2n}\otimes\mathcal{L}_{2n})\otimes F(W')}}$. This completes the proof of the proposition.
\end{proof}

It is almost immediate from Propositions \ref{prop:Ocpq0-first-properties} and \ref{prop:Ocpq0-contragredients} that $\mathcal{O}_{c_{p,q}}^0$ is a locally finite abelian category. Indeed, since it is a full subcategory of the locally finite abelian category $\mathcal{O}_{c_{p,q}}$ which contains $0$ and is closed under finite direct sums and quotients, all that remains is to show that~\smash{$\mathcal{O}_{c_{p,q}}^0$} is closed under submodules. But if $\widetilde{W}$ is a submodule of an object $W$ of \smash{$\mathcal{O}_{c_{p,q}}^0$}, then $\widetilde{W}'$ is a quotient of $W'$ and thus is an object of \smash{$\mathcal{O}_{c_{p,q}}^0$}, and then so is $\widetilde{W}\cong\bigl(\widetilde{W}'\bigr)'$. Moreover, by~\cite{ALSW}, contragredient modules give \smash{$\mathcal{O}_{c_{p,q}}^0$} the structure of a ribbon Grothendieck--Verdier category, which means that taking contragredients is a contravariant auto-equivalence of \smash{$\mathcal{O}_{c_{p,q}}$} and that there is a natural isomorphism
$\operatorname{Hom}(W\boxtimes X,\mathcal{K}_{1,1}')\cong\operatorname{Hom}(W,X')$
for all objects $W$ and $X$ in~$\mathcal{O}_{c_{p,q}}$.

\begin{Theorem}
$\mathcal{O}_{c_{p,q}}^0$ is a locally finite abelian ribbon Grothendieck--Verdier category.
\end{Theorem}

In \cite[Section 7]{MS-cpq-Vir}, we conjectured that there should be a suitable tensor subcategory of $\mathcal{O}_{c_{p,q}}$ that contains all simple objects of $\mathcal{O}_{c_{p,q}}$ and has enough projectives, and that such a subcategory should be the right category for constructing a full (bulk) logarithmic conformal field theory based on the Virasoro algebra, that is, a logarithmic minimal module at $c_{p,q}$ central charge. We now conjecture that \smash{$\mathcal{O}_{c_{p,q}}^0$} is the appropriate subcategory of \smash{$\mathcal{O}_{c_{p,q}}$} for logarithmic conformal field theory. Although we have not yet shown that $\mathcal{O}_{c_{p,q}}^0$ has enough projectives, it may be possible to prove this using the existence and structure of projective objects in $\operatorname{Rep}(W_{p,q})$ \cite{Nak-Wpq-projective}. We leave the problem of obtaining projective objects in $\mathcal{O}_{c_{p,q}}^0$ to future work.

\section{Conclusion and outlook}

In this paper, we have given a new tensor-categorical construction of the triplet $W$-algebra~$W_{p,q}$ for coprime $p,q\in\mathbb{Z}_{\geq 2}$. Specifically, we have ``glued'' $\operatorname{Rep} \mathrm{PSL}_2$ with a subcategory of $\mathcal{V}{\rm ir}$-modules at central charge $c_{p,q}$ having $\mathrm{PSL}_2$-fusion rules, and we have then appropriately modified to obtain the non-simple VOA $W_{p,q}$. Due to the involvement of $\operatorname{Rep} \mathrm{PSL}_2$ in the construction, a major corollary is that the automorphism group of $W_{p,q}$ is $\mathrm{PSL}_2(\mathbb{C})$, with no need to use complicated analysis of screening operators to give an explicit action of $\mathrm{PSL}_2(\mathbb{C})$. It would be interesting to explore whether the $\mathrm{PSL}_2(\mathbb{C})$-action on $W_{p,q}$ could be exploited to simplify the proofs from \cite{TW2} of important properties of $W_{p,q}$ such as $C_2$-cofiniteness and the classification of its simple modules.

We have also defined a tensor subcategory $\mathcal{O}_{c_{p,q}}^0$ of $\mathcal{V}{\rm ir}$-modules at central charge $c_{p,q}$ that induce to ordinary modules for $W_{p,q}$, and we have shown that it contains all simple objects of $\mathcal{O}_{c_{p,q}}$ and is closed under contragredients and thus is a ribbon Grothendieck--Verdier category. The main remaining open problems for $\mathcal{O}_{c_{p,q}}^0$ are to show that it has enough projective objects, and to explore its applications in logarithmic conformal field theory. For the latter problem, our result that $\mathcal{O}_{c_{p,q}}^0$ is a Grothendieck--Verdier category will be key since there is now a theory of module categories and Frobenius algebras for Grothendieck--Verdier categories under development \cite{FSSW}, and these structures are important in constructions of full CFTs \cite{RFFS, RGW}.

Beyond conformal field theory, topological quantum field theories and invariants of low-dimensional manifolds can be constructed from braided tensor categories that are not necessarily semisimple (see, for example, \cite{BGR,Bla-De-Renzi, CCC, DR-G-G-PM-R} for some recent results), which is significant since for example
semisimple $4$-dimensional topological field theories cannot detect exotic smooth structures \cite{Reutter}. While most results in this direction assume rigidity for the braided tensor categories under consideration, it would also be interesting to explore how much these results generalize to non-rigid categories such as $\operatorname{Rep}(W_{p,q})$ and $\mathcal{O}_{c_{p,q}}^0$, or to the module categories for universal affine $\mathfrak{sl}_2$ VOAs studied in~\cite{MY-univ-sl2}.

The triplet algebras $W_{p,q}$ are special cases of a large class of VOAs sometimes called Feigin--Tipunin algebras \cite{FT, Sugimoto}. In general, these VOAs are (or are expected to be) large extensions of an affine $W$-algebra associated to a simple Lie algebra at a given level, such that the affine $W$-subalgebra is the fixed points of a corresponding Lie group of automorphisms. Thus $W_{p,q}$ is the Feigin--Tipunin algebra associated to the principal affine $W$-algebra of $\mathfrak{sl}_2$ at level $-2+\frac{p}{q}$. It would be interesting to explore whether the methods of this paper could be used to study Feigin--Tipunin algebras beyond $W_{p,q}$. In fact, we expect that the methods of Section~\ref{sec:comm-alg-and-VOA} can be combined with the results of~\cite{MY-univ-sl2} to give a tensor-categorical construction of the Feigin--Tipunin algebra associated to the universal affine VOA of $\mathfrak{sl}_2$ (the affine $W$-algebra of $\mathfrak{sl}_2$ for the trivial nilpotent element) at level $-2+\frac{p}{q}$ for coprime $p,q\in\mathbb{Z}_{\geq 1}$; when $p=1$, these algebras have been recently studied in~\cite{CNS}.

Finally, we note that there are vertex operator superalgebra analogues of the triplet $W$-algebras which are extensions of the $N=1$ super Virasoro vertex operator superalgebra \cite{AM-N=1-triplet}. We expect that our methods could also apply to give tensor-categorical constructions of these algebras, using the $N=1$ super Virasoro tensor categories recently constructed in \cite{CMOY}.

\subsection*{Acknowledgments}

R.M.\ is partially supported by a research fellowship from the Alexander von Humboldt Foundation and also thanks the Universit\"{a}t Hamburg for hospitality during the visit in which part of this work was done. We also thank Hiromu Nakano and Simon Wood for helpful discussions, and we thank the referees for comments and suggestions.

\pdfbookmark[1]{References}{ref}
\LastPageEnding

\end{document}